\documentclass[final,1p,times]{elsarticle}

\usepackage{epsfig}
\usepackage{amssymb}
\usepackage{amsthm}
\usepackage{amsfonts}
\usepackage{amsmath}
\usepackage{color}
\biboptions{numbers,sort&compress}
\newtheorem{theorem}{Theorem}[section]
\newtheorem{lemma}{Lemma}[section]
\newtheorem{proposition}{Proposition}[section]
\newtheorem{corollary}{Corollary}[section]

\newdefinition{remark}{Remark}[section]

\begin{document}

\begin{frontmatter}



\title{Indefinite linear quadratic  control of mean-field backward stochastic differential equation}



\author[1,2]{Wencan Wang}           \ead{wwencan@163.com}
\author[3]{Huanjun Zhang\corref{cor1}}                \ead{zhhuanjun@163.com}
\cortext[cor1]{Corresponding author}
\address[1]{School of Mathematical and Physical Sciences, Wuhan Textile University, Wuhan 430200, Hubei,  PR China}
\address[2]{School of Control Science and Engineering, Shandong University, Jinan 250061, Shandong,   PR China}
\address[3]{School of Mathematics and Statistics, Shandong Normal  University, Jinan 250014, Shandong,  PR China}

\begin{abstract}
This paper is concerned with a  general linear quadratic (LQ) control problem of mean-field  backward stochastic differential equation (BSDE).
Here, the weighting matrices  in the cost functional are allowed to be indefinite.   Necessary and sufficient conditions for optimality are obtained via a mean-field forward-backward stochastic differential equation (FBSDE).
By investigating the connections with LQ problems of mean-field forward systems and taking some limiting procedures, we establish the solvabilities of corresponding Riccati equations in the case that  cost functional is uniformly convex. Subsequently,  an explicit formula of   optimal control and  optimal cost  are derived. Moreover, some sufficient conditions for the uniform convexity of  cost functional are also proposed in terms of Riccati equations, which have not been considered in existing literatures for backward systems. Some examples are provided to
illustrate our results.
\end{abstract}




\end{frontmatter}


\section{Introduction}
In recent years, there is an increasing interest on mean-field type stochastic control problems due to their  wide applications in mathematics, engineering and finance. Different from classical  stochastic control problems, a new feature of the  problems is that both stochastic system and cost functional embody the state and the control along with their statistical distributions, which provides a simple but effective technique  for describing  individuals' mutual interactions. In this case, the controlled system is an SDE of McKean-Vlasov type, which  was first introduced  by Kac \cite{Kac1956}. Since then, there have been considerable results on related topics. Interested readers may
refer to \cite{BDLP2009,BLP2009} for  the theory of  mean-field BSDEs,  \cite{AD2011,BLM2016,WXX2017,WZZ2013,WW2022} for various
versions of stochastic maximum principles for mean-field models. The past few years have also witnessed rapid development of mean-field LQ control theory, both in finite and infinite horizons \cite{Yong2013,HLY2017,ELN2013,NEL2015}.
Yong \cite{Yong2013} investigated a  mean-field LQ problem in finite horizon systemically by using  a decoupling technique. A feedback  optimal control is characterized by  two coupled  Riccati equations.  Further in \cite{HLY2017}, Huang et. al. studied an infinite horizon case, where an optimal control is expressed via    two coupled algebraic Riccati equations. Elliott et. al. \cite{ELN2013} and Ni
et. al. \cite{NEL2015} dealt with discrete time mean-field LQ problems for  finite horizon  and infinite horizon, respectively.  We remark that all the aforementioned papers followed the  standard assumption, that is the weighting matrices in  cost functional are imposed with some positive-definite conditions.

The study of indefinite LQ control problems can be  traced back to  Chen et. al. \cite{CLZ1998}, who first pointed out that some stochastic LQ problems with indefinite (in particular, negative) control weighting matrices may still be sensible and well-posed. 
An equivalent cost functional method was introduced in \cite{Yu2013} to study an indefinite LQ control problem, which was extended to  mean-field case in \cite{WW2021}. Li et. al. \cite{LLY2020} studied an indefinite LQ problem of mean-field SDE by introducing  relax compensators. The open-loop and closed-loop solvabilities for LQ control problems of mean-field type are discussed in \cite{sun2017} and  \cite{LSY2016}, respectively. It was shown that the standard assumption is not  necessary   for the solvabilities of mean-field LQ control problems.
Ni et. al. \cite{NLZ2015,NZL2014} considered   discrete time indefinite LQ optimal control problems of mean-field type with   infinite horizon and  finite horizon, respectively. Wang and Zhang \cite{WZ2021} investigated uniform stabilization and asymptotic optimality for indefinite mean-field LQ social control system with multiplicative noises. We point out that  the existing literatures on indefinite LQ control problems focused on forward stochastic systems.

In financial investment, a European contingent claim $\zeta$, which is a random variable, can be regarded as a contract to be guaranteed at maturity $T$.
One frequently encounters  the situation that funds may be injected or withdrawn from the replication process of a contingent claim so as to achieve some other goals.
This naturally results in stochastic control problems of  BSDEs \cite{LSX2019, DHW2018,FHW2021,HWW2016}.
Li et. al. \cite{LSX2019} considered an LQ control problem of mean-field BSDE, where  explicit formulas of optimal control and optimal cost are derived. Mean-field games and social optimal problems of backward stochastic systems with LQ formulation are discussed in \cite{DHW2018,FHW2021,HWW2016}.
However, the study on indefinite LQ control problems of BSDEs is quite lacking in literatures. To our best knowledge, there are only a few papers on this type of control problems, including \cite{SWX2019} and \cite{SWX2022}, which constructed the optimal control in terms of a Riccati equation, an adjoint equation and a BSDE under the uniform convexity  of cost functional for homogeneous and nonhomogeneous cases, respectively. 
Indefinite LQ control problems of mean-field BSDEs are underdeveloped in literatures and  many fundamental questions remain unsolved.

This paper is concerned with  an indefinite LQ control problem of mean-field BSDE.
As  preliminary results, we  derive necessary and sufficient conditions
for  optimality. The optimal control is characterized via a mean-field FBSDE,   together with a stationarity condition. By investigating the  connections
with forward mean-field LQ problems, we introduce some reductions of Problem (BLQ), which leads to an equivalent control problem with \eqref{eq: simcoeff} holds. Combining with  this transformation and some limiting procedures, we derive  the   solvabilities of  Riccati equations associated with Problem (BLQ).
With these results, we construct an  explicit formulas of  optimal control and optimal cost under the uniform convexity of cost functional in terms  of Riccati equations, an adjoint process and a mean-field BSDE. Finally,  by empolying the equivalent cost functional method, we proposed some sufficient conditions for the uniform convexity of cost functional in terms of Riccati equations.

Our work distinguishes itself from existing literatures in the following aspects.
(i) We develop a general procedure of constructing optimal control of Problem (BLQ) under  the uniform convexity of cost functional (Assumption $A3$), which extends the results in \cite{SWX2023,SWX2022}. With the appearance of mean-field term, we need to introduce two coupled Riccati equations \eqref{eq.RE1}-\eqref{eq.RE2}, which play   important roles in deriving explicit formulas of  optimal control and optimal cost. 
Moreover, it takes more efforts to establish the connections between Problem (BLQ) and Problem (FLQ$_{\lambda}$) as well as simplify Problem (BLQ), due to our mean-field setting.
(ii) In this paper,  a general indefinite  LQ control problem of mean-field BADE is considered, where the cross-product terms of control and state processes are involved in cost functional. Moreover, both the state equation and the cost functional contain the nonhomogeneous terms.
As we can  see in Section 3, the cross-product terms and nonhomogeneous terms bring lots of difficulties in deriving  optimal control and optimal cost  of Problem (BLQ). Compared with \cite{LSX2019}, the   representations of  solution $Z$ and  optimal control $u$ in Theorem \ref{Th: no-feedback} deserves  more complex structures.
(iii) We provide some necessary and sufficient conditions for the uniform convexity of cost functional in terms of Riccati equations \eqref{eq.RE1} and \eqref{eq.RE2}. 
Due to the special structure of  backward systems, we  adopt the equivalent cost functional method  to prove the sufficiency in our paper, which is quite different  from mean-field LQ control problems of forward stochastic systems. Moreover, as we will see in
Example 6.2, the equivalent cost functional method also provides an alternative and effective way to obtain the uniform convexity of cost functional.
In fact, \eqref{eq:equi} implies that  the uniform convexity is equivalent for a family of  cost functional $J_{h}(\zeta;  u)$.
Thus, in
order to obtain the uniform convexity of cost functional for original control problem,
we need only to find an equivalent cost functional satisfying Assumption A3.

The rest of this paper is organized as follows.  In Section 2, we formulate an indefinite LQ control problem of mean-field BSDE and  give some preliminary results. Section 3 aims to derive explicit formulas  of optimal control and optimal cost in the case that the cost functional is uniformly convex. In Section 4, we  devoted to giving sufficient conditions for the uniform  convexity of  cost functional in terms of Riccati equations. Section 5 gives
several illustrative examples. Finally, we
conclude this paper.
\section{Problem formulation and preliminaries}
Throughout this paper,  $\mathbb R^{n\times m}$ denotes  the set of all $n\times m$   matrices and   $\mathbb S^n$  denotes  the set of all $n\times n$ symmetric matrices. 
In particular, we mean by $S\geq 0$ $(S>0)$ if $S$  is a nonnegative  (positive) definite matrix. For a matrix valued
function $S: [0, T] \to  \mathbb S^n$, we mean by $S\gg 0$  that $S(t)$ is uniformly positive,
i. e., there is a positive real number $\alpha$ such that $S(t) \geq \alpha I$ for any $t \in  [0, T ]$. For a matrix $M\in \mathbb R^{n\times m}$, let $M^\top$ be its transpose. The inner product $\langle\cdot,\cdot\rangle$ on $\mathbb R^{n\times m}$ is defined by $\langle M,N\rangle\mapsto tr(M^\top N)$ with an  induced norm  $|M|=\sqrt{tr(M^\top M)}$. Let $(\Omega,\mathcal F, \mathbb F, \mathbb P)$ be a complete filtered probability space and let $T>0$ be a fixed time horizon.  $W(t): 0\leq t\leq T$ is  an  $\mathbb R$-valued standard Wiener process, defined on $(\Omega,\mathcal F, \mathbb F, \mathbb P)$. $\mathbb F\equiv\{\mathcal F_t\}_{t\geq0}$ is a natural filtration of $W(\cdot)$ augmented by all   $\mathbb P$-null sets. 
For any Euclidean space $\mathbb{M}$, we give the following notations:\\
$\mathcal L_{\mathcal F_T}^2(\Omega; \mathbb{M})=\Big\{\zeta: \Omega\to \mathbb{M}| \zeta $ is an $\mathcal F_T$-measurable random variable, $\mathbb E[|\zeta|^2]<\infty$\Big\};\\
$\mathcal L^\infty(0,T;\mathbb{M})=\Big\{v:[0,T]\to \mathbb{M}|v$ is a bounded and deterministic  function\Big\};\\
$\mathcal L_{\mathbb F}^2(0,T;\mathbb{M})=\bigg\{v:[0,T]\times \Omega\to \mathbb{M}|v$  is an $\mathbb F$-progressively measurable stochastic process, $\mathbb E\left[\int_{0}^T|v(t)|^2dt\right]<\infty\bigg\};$\\
$\mathcal S_{\mathbb F}^2(0,T;\mathbb{M})=\bigg\{v:[0,T]\times \Omega\to  \mathbb{M}|v$ is an $\mathbb F$-progressively measurable stochastic process and has continuous paths, $\mathbb E\left[\sup_{t\in[0,T]}|v(t)|^2\right]<\infty\bigg\}$.

Consider a controlled mean-field linear BSDE
\begin{equation}\label{station 1}
\left\{ \begin{aligned}
dY(t)=&\big[A(t)Y(t)+\widehat{A}(t)\mathbb E[Y(t)]+B(t)u(t)+\widehat{B}(t)\mathbb E[u(t)]+C(t)Z(t)+\widehat{C}(t)\mathbb E[Z(t)]+f(t)\big]dt\\&+Z(t)dW(t),  \\
Y(T)=&\ \zeta,
\end{aligned}\right.
\end{equation}
where $\zeta\in \mathcal{L}_{\mathcal F_T}^2(\Omega; \mathbb R^n)$ and $u(\cdot)$, valued in $\mathbb R^m$, is a control process.
Introduce an admissible control set $\mathcal U[0, T]=\mathcal L_{\mathbb F}^2(0, T;\mathbb{R}^m)$.
Any $u\in \mathcal U[0, T]$ is called an  admissible control.\\
\textbf{Assumption $A1$:} The coefficients of  system \eqref{station 1} satisfy
\begin{equation*}
A(\cdot),  \widehat{A}(\cdot), C(\cdot), \widehat{C}(\cdot)\in\mathcal L^\infty(0,T;\mathbb R^{n\times n}),\
 B(\cdot), \widehat{B}(\cdot) \in\mathcal L^\infty(0,T;\mathbb R^{n\times m}), f\in \mathcal L_{\mathbb F}^2(0, T;\mathbb R^n).
\end{equation*}
Under Assumption $A1$,  system   \eqref{station 1} admits a unique solution pair  $(Y,  Z)\in\mathcal S_{\mathbb F}^2(0, T;\mathbb R^n)\times\mathcal L_{\mathbb F}^2(0, T;\mathbb R^n)$, which is called the corresponding state process, for any $u\in \mathcal U[0, T]$ (see Li et. al. \cite{LSX2019}).
We introduce a quadratic cost functional
\begin{equation}\label{cost functional}
\begin{aligned}
J(\zeta;  u)=&\frac{1}{2}\mathbb E\left[\langle GY(0),  Y(0)\rangle
+\langle \widehat{G}\mathbb E[Y(0)],  \mathbb E[Y(0)]\rangle+2\langle g,  Y(0)\rangle\right.\\&+2\int_{0}^T\Bigg\langle\left(\begin{array}{c}q(t)\\\rho_1(t)\\\rho_2(t)\end{array}\right),  \left(\begin{array}{c}Y(t)\\Z(t)\\u(t)\end{array}\right)\Bigg\rangle dt
\\&+\int_{0}^T\Bigg\langle\left(\begin{array}{ccc}
Q(t)&S_{1}(t)&S_{2}(t)
\\S_{1}^\top(t)&R_{11}(t)&R_{12}(t)
\\S_{2}^\top(t)&R_{12}^\top(t)&R_{22}(t) \end{array}\right)\left(\begin{array}{c}Y(t)\\Z(t)\\u(t)\end{array}\right),  \left(\begin{array}{c}Y(t)\\Z(t)\\u(t)\end{array}\right)\Bigg\rangle dt
\\&\left.+\int_{0}^T\Bigg\langle\left(\begin{array}{ccc}
\widehat{Q}(t)&\widehat{S}_{1}(t)&\widehat{S}_{2}(t)
\\\widehat{S}_{1}^\top(t)&\widehat{R}_{11}(t)&\widehat{R}_{12}(t)
\\\widehat{S}_{2}^\top(t)&\widehat{R}_{12}^\top(t)&\widehat{R}_{22}(t) \end{array}\right)\left(\begin{array}{c}\mathbb E[Y(t)]\\\mathbb E[Z(t)]\\\mathbb E[u(t)]\end{array}\right),  \left(\begin{array}{c}\mathbb E[Y(t)]\\\mathbb E[Z(t)]\\\mathbb E[u(t)]\end{array}\right)\Bigg\rangle dt\right],
\end{aligned}
\end{equation}
where $g$ is an $\mathcal F_0$-measurable random variable satisfying $\mathbb E[|g|^2] < \infty$.
For simplicity, we may suppress time $t$  and use the following notations in this paper:
\begin{equation*}
\begin{aligned}
& S=\left(\begin{array}{cc}S_1&S_2\end{array}\right), \ \widehat{S}=\left(\begin{array}{cc}\widehat{S}_1&\widehat{S}_2\end{array}\right),
R=\left(\begin{array}{cc}R_{11}&R_{12}\\R_{12}^\top&R_{22}\end{array}\right),  \ \widehat{R}=\left(\begin{array}{cc}\widehat{R}_{11}&\widehat{R}_{12}
\\\widehat{R}_{12}^\top&\widehat{R}_{22}\end{array}\right),
\\&\widetilde{P}=P+\widehat{P},  \text{ for } P=A,  B,  C,  Q,  S_1,  S_2,  S,  R_{11},  R_{12},  R_{22},  R,   G.
\end{aligned}
\end{equation*}
\textbf{Assumption $A2$:} The weighting matrices in  cost functional \eqref{cost functional} satisfy
\begin{equation*}
\left\{\begin{aligned}
&G,  \widehat{G}\in \mathbb S^{n}, \ Q,  \widehat{Q}\in\mathcal L^\infty(0, T;\mathbb S^{n}),  \ S_1,  \widehat{S}_1\in\mathcal L^\infty(0, T;\mathbb R^{n\times n}),  \ S_2,  \widehat{S}_2\in\mathcal L^\infty(0, T;\mathbb R^{n\times m}),
\\&R_{11},  \widehat{R}_{11}\in\mathcal L^\infty(0, T;\mathbb S^{n}),  R_{12},  \widehat{R}_{12}\in\mathcal L^\infty(0, T;\mathbb R^{n\times m}), R_{22},  \widehat{R}_{22}\in\mathcal L^\infty(0, T;\mathbb S^{m}),
\\&q\in \mathcal L_{\mathbb F}^2(0, T;\mathbb R^n), \rho_1\in\mathcal L_{\mathbb F}^2(0, T;\mathbb R^n), \rho_2\in\mathcal L_{\mathbb F}^2(0, T;\mathbb R^m).
\end{aligned}\right.
\end{equation*}
Our mean-field backward stochastic LQ control problem   can be stated as follows.
\\\textbf{Problem (BLQ)}. Find a $u^*\in \mathcal U[0, T]$ such that
\begin{equation}\label{inf}
V(\zeta)=J(\zeta;  u^*)=\inf_{u\in \mathcal U[0, T]}J(\zeta;  u).
\end{equation}
Any $u^*\in \mathcal U[0,T]$ satisfying \eqref{inf} is called an optimal control of Problem (BLQ), and the corresponding state process $(Y^*,Z^*)$   is called an optimal state process. $(Y^*,Z^*, u^*)$ is called an optimal triple. In the special case that  $f, g, q, \rho_1, \rho_2$ vanish, we
denote the corresponding Problem (BLQ) by Problem (BLQ)$_0$. The corresponding cost functional is denoted by $J_0(\zeta; u)$. Under Assumption  $A2$,  cost functional \eqref{cost functional} is well-defined.  
The following theorem gives a characterization of optimal control for Problem (BLQ).
\begin{theorem}\label{necessary condition}
Suppose that  Assumptions $A1$-$A2$ hold and $\zeta\in\mathcal L_{\mathcal F_T}^2(\Omega; \mathbb R^n)$ is given.  $u^*\in \mathcal U[0, T]$ is  an optimal control of Problem (BLQ)   if and only if the following conditions hold:\\
(i) $J_0(0; u)\geq 0$, $\forall u\in \mathcal U[0,T]$.
\\(ii) The mean-field FBSDE
\begin{equation}\label{eq: Hamiltonian}
\left\{\begin{aligned}
dX^*=&-\left(A^\top X^*+\widehat{A}^\top \mathbb E[X^*]+QY^*+\widehat{Q}\mathbb E[Y^*]+S_1Z^*+\widehat{S_1} \mathbb E[Z^*]+S_2u^*+\widehat{S}_2\mathbb E[u^*]+q\right)dt
\\&-\left(C^\top X^*+\widehat{C}^\top \mathbb E[X^*]+S_1^\top Y^*+\widehat{S}_1^\top \mathbb E[Y^*]+R_{11} Z^*+\widehat{R}_{11}\mathbb E[Z^*]+R_{12}u^*+\widehat{R}_{12}\mathbb E[u^*]+\rho_1\right)dW,
\\dY^*=&\ \big(AY^*+\widehat{A}\mathbb E[Y^*]+Bu^*+\widehat{B}\mathbb E[u^*]+CZ^*+\widehat{C}\mathbb E[Z^*]+f\big)dt+Z^*dW,
\\X^*(0)=&-GY^*(0)-\widehat G\mathbb E[Y^*(0)]-g,     Y(T)= \zeta 
\end{aligned}
\right.
\end{equation}
admits unique solution $(X^*,  Y^*,  Z^*)$, such that
\begin{equation*}
B^\top X^*+\widehat{B}^\top \mathbb E[X^*]+S_2^\top Y^*+\widehat{S}_2^\top \mathbb E[Y^*]+R_{12}^\top Z^*+\widehat{R}_{12}^\top \mathbb E[Z^*]+R_{22}u^*+\widehat{R}_{22}\mathbb E[u^*]+\rho_2=0.
\end{equation*}
\end{theorem}
\begin{proof}
$u^*\in\mathcal U[0,T]$ is an optimal  control of  Problem (BLQ) if and only if
$$J(\zeta; u^*+\varepsilon u)-J(\zeta; u^*)\geq 0, \forall u\in\mathcal U[0,T], \varepsilon\in \mathbb R.$$
Let $(Y^\varepsilon, Z^\varepsilon)$   be the  solution to  system \eqref{station 1} corresponding to $u^{\varepsilon}=u^*+\varepsilon u$.  It is  clearly that $(Y^\varepsilon, Z^\varepsilon)=(Y^*, Z^*)+\varepsilon(\delta Y, \delta Z)$, where $(\delta Y, \delta Z)$ satisfies
\begin{equation*}
\left\{ \begin{aligned}
d\delta Y=&\Big(A\delta Y+\widehat{A}\mathbb E[\delta Y]+Bu+\widehat{B}\mathbb E[u]+C\delta Z+\widehat{C}\mathbb E[\delta Z]\Big)dt+\delta ZdW, \\
\delta Y(T)=&\ 0.
\end{aligned}\right.
\end{equation*}
Consequently, we have
\begin{equation*}
\begin{aligned}
&J(\zeta; u^{\varepsilon})-J(\zeta; u^*)
\\=&\ \frac{\varepsilon^2}{2}J_0(0; u)+\varepsilon\mathbb E[ \langle g, \delta Y(0)\rangle]
+\varepsilon\mathbb E[\langle G\delta Y(0), Y^*(0)\rangle]+\varepsilon\mathbb E[\langle \widehat{G}\mathbb E[\delta Y(0)],\mathbb E[Y^*(0)]\rangle]
\\&+\varepsilon\mathbb E\left[\int_{0}^T\left \langle \left(\begin{array}{ccc}
Q&S_1&S_2\\
S_1^{\top}&R_{11}& R_{12}\\S_2^{\top}& R_{12}^{\top}&R_{22}\end{array}
\right)\left(\begin{array}{ccc}
\delta Y\\
 \delta Z\\ u\end{array}
\right), \left(\begin{array}{ccc}
Y^*\\
 Z^*\\   u^*\end{array}
\right)\right\rangle dt\right]+\varepsilon
\mathbb E\left[\int_{0}^T\left\langle\left(\begin{array}{ccc}
q\\
 \rho_1\\  \rho_2\end{array}
\right), \left(\begin{array}{ccc}
\delta Y\\
 \delta Z\\   u\end{array}
\right)\right\rangle dt\right]
\\&+\varepsilon\mathbb E\left[\int_{0}^T\left \langle \left(\begin{array}{ccc}
\widehat{Q}&\widehat S_1&\widehat S_2\\
\widehat S_1^{\top}&\widehat R_{11}& \widehat R_{12}\\\widehat S_2^{\top}& \widehat R_{12}^{\top}&\widehat R_{22}\end{array}
\right)\left(\begin{array}{ccc}
\mathbb E[\delta Y]\\
 \mathbb E[\delta Z]\\ \mathbb E[u]\end{array}
\right), \left(\begin{array}{ccc}
\mathbb E[Y^*]\\
 \mathbb E[Z^*]\\   \mathbb E[u^*]\end{array}
\right)\right\rangle dt\right]
\\=&\ \varepsilon\mathbb E\left[\int_{0}^T\langle B^\top X^*+\widehat{B}^\top \mathbb E[X^*]+S_2^\top Y^*+\widehat{S}_2^\top \mathbb E[Y^*]+R_{12}^\top Z^*+\widehat{R}_{12}^\top \mathbb E[Z^*]+R_{22}u^*+\widehat{R}_{22}\mathbb E[u^*]+\rho_2, u\rangle\right]
\\&+\frac{\varepsilon^2}{2}J_0(0; u).
\end{aligned}
\end{equation*}
From the above arguments, it is easy to draw the conclusions.
\end{proof}

\section{The optimal control and Riccati equations}
From the above arguments, it is difficult to check whether Problem (BLQ) admits a unique optimal control when  $J_0(0; u)\geq 0, \forall u\in \mathcal U[0,T]$, due to the fact that the unique solvability of mean-field FBSDE  \eqref{eq: Hamiltonian} is hard  to tackle without the positive definiteness assumptions on the weighting matrices.
In the following, we introduce a condition slightly stronger than $J_0(0; u)\geq 0$, $\forall u\in \mathcal U[0,T]$, which enables us to overcome the challenge brought by the indefiniteness of weighting matrices.\\
\textbf{Assumption $A3$:} There exists a constant $\alpha>0$, such that
$$J_0(0, u)\geq\alpha \mathbb E\left[\int_{0}^T|u|^2dt\right],  \forall u\in \mathcal U[0,T].$$
Actually, Assumption $A3$ is also called the uniform convexity of cost functional, which is sufficient for  the existence and uniqueness of optimal control of Problem (BLQ). In this section, we will give a more   specific characterization of optimal control  and construct an explicit formula of the corresponding  optimal cost under Assumption $A3$ via two Riccati equations with terminal conditions, a mean-field BSDE.

\subsection{Connections with  LQ problems of  mean-field  forward systems}
Consider a controlled mean-field forward system
\begin{equation}\label{eq. Fstate}
\left\{\begin{aligned}
dX=&\ \big(AX+\widehat{A}\mathbb E[X]+Bu+\widehat{B}\mathbb E[u]+Cv+\widehat{C}\mathbb E[v]\big)dt+vdW,
\\X(0)=&\ \xi,  
\end{aligned}
\right.
\end{equation}
where $\xi$ is an $\mathcal F_{0}$-measurable random variable satisfying $\mathbb E[|\xi|^2]<\infty$. The control process is the pair $(u,v)\in  \mathcal L_{\mathbb F}^2(0, T;\mathbb{R}^m)\times\mathcal L_{\mathbb F}^2(0, T;\mathbb{R}^n)=\mathcal U[0, T]\times\mathcal V[0, T]$.
Introduce a quadratic cost functional
\begin{equation} \label{eq. Fcost}
\begin{aligned}
\mathcal{J}_{\lambda}(\xi;  u,  v)=&\frac{1}{2}\mathbb E\left[\lambda\langle X(T),  X(T)\rangle+\int_{0}^T\Bigg\langle\left(\begin{array}{ccc}
Q&S_{1}&S_{2}
\\S_{1}^\top&R_{11}&R_{12}
\\S_{2}^\top&R_{12}^\top&R_{22} \end{array}\right)\left(\begin{array}{c}X\\v\\u\end{array}\right),  \left(\begin{array}{c}X\\v\\u\end{array}\right)\Bigg\rangle dt\right.
\\&\left.+\int_{0}^T\Bigg\langle\left(\begin{array}{ccc}
\widehat{Q}&\widehat{S}_{1}&\widehat{S}_{2}
\\\widehat{S}_{1}^\top&\widehat{R}_{11}&\widehat{R}_{12}
\\\widehat{S}_{2}^\top&\widehat{R}_{12}^\top&\widehat{R}_{22} \end{array}\right)\left(\begin{array}{c} \mathbb E[X]\\\mathbb E[v]\\\mathbb E[u]\end{array}\right),  \left(\begin{array}{c}\mathbb E[X]\\\mathbb E[v]\\\mathbb E[u]\end{array}\right)\Bigg\rangle dt\right],
\end{aligned}
\end{equation}
where $\lambda>0$ is a constant parameter.
We pose an  LQ  problem associate with  system \eqref{eq. Fstate} and  cost functional \eqref{eq. Fcost} as follows.\\
\textbf{Problem (FLQ$_\lambda$)}: Find a $(u^*,  v^*)\in\mathcal U[0, T]\times\mathcal V[0, T]$, such that
\begin{equation*}
\mathcal{V}_{\lambda}(\xi)=\mathcal{J}_{\lambda}(\xi;  u^*,  v^*)=\inf_{(u,  v)\in\mathcal U[0, T]\times\mathcal V[0, T]}\mathcal{J}_{\lambda}(\xi;  u,  v).
\end{equation*}
We now give the following result, which indicates some connections between Problem (BLQ) and Problem (FLQ$_\lambda$).
\begin{theorem}\label{th: connection}
Let Assumptions $A1$-$A3$ hold. Then there exist   constants $\rho>0$, and $\lambda_0>0$, such that for any $\lambda\geq\lambda_0$,
$$\mathcal{J}_\lambda(0; u,  v)\geq\rho \mathbb E\left[\int_{0}^T(|u|^2+|v|^2)dt\right],   \forall (u, v)\in \mathcal U[0, T]\times\mathcal V[0, T].$$
If, in addition, $-G\geq 0$, $-\widehat G\geq 0$, then for $\lambda\geq\lambda_0$,
$$\mathcal{J}_\lambda(\xi;  u,  v)\geq \rho \mathbb E\left[\int_{0}^T(|u|^2+|v|^2)dt\right],   \forall (u, v)\in \mathcal U[0, T]\times\mathcal V[0, T].$$
\end{theorem}
\begin{proof}
For any $(u,v)\in \mathcal U[0,T]\times\mathcal V[0,T]$, let $X$ be the corresponding solution of \eqref{eq. Fstate}. It is obvious that $\zeta\triangleq X(T)\in \mathcal L_{\mathcal F_T}^2(\Omega; \mathbb R^n)$. Regard $(X, u)$ as the solution of
\begin{equation*}
\left\{\begin{aligned}
dX=&\ \big(AX+\widehat{A}\mathbb E[X]+Bu+\widehat{B}\mathbb E[u]+Cv+\widehat{C}\mathbb E[v]\big)dt+vdW,
\\X(T)=&\ \zeta.  
\end{aligned}
\right.
\end{equation*}
Further, we introduce the following two mean-field BSDEs:
\begin{equation*}
\left\{\begin{aligned}
dy_0=&\ \big(Ay_0+\widehat{A}\mathbb E[y_0]+Bu+\widehat{B}\mathbb E[u]+Cz_0+\widehat{C}\mathbb E[z_0]\big)dt+z_0dW,
\\y_{0}(T)=&\ 0, 
\end{aligned}
\right.
\end{equation*}
and
\begin{equation*}
\left\{\begin{aligned}
dy_1=&\ \big(Ay_1+\widehat{A}\mathbb E[y_1]+Cz_1+\widehat{C}\mathbb E[z_1]\big)dt+z_1dW,
\\y_{1}(T)=&\ \zeta.  
\end{aligned}
\right.
\end{equation*}
Thus we have $$X=y_0+y_1, v=z_0+z_1.$$
Denote
$$M=\left(\begin{array}{ccc}
Q&S_{1}&S_{2}
\\S_{1}^\top&R_{11}&R_{12}
\\S_{2}^\top&R_{12}^\top&R_{22} \end{array}\right),  \widehat M=\left(\begin{array}{ccc}
\widehat{Q}&\widehat{S}_{1}&\widehat{S}_{2}
\\\widehat{S}_{1}^\top&\widehat{R}_{11}&\widehat{R}_{12}
\\\widehat{S}_{2}^\top&\widehat{R}_{12}^\top&\widehat{R}_{22} \end{array}\right),  \theta_0=\left(\begin{array}{c}y_0\\z_0\\u\end{array}\right),  \theta_1=\left(\begin{array}{c}y_1\\z_1\\0\end{array}\right).$$
With these notations,
\begin{equation*}
\begin{aligned}
J_0(0;  u)=&\frac{1}{2}\mathbb E\left[\langle Gy_0(0),  y_0(0)\rangle+\langle \widehat{G}\mathbb E[y_0(0)],  \mathbb E[y_0(0)]\rangle\right.\\&\left.+\int_{0}^T\langle M\theta_0,  \theta_0\rangle dt+\int_{0}^T\big\langle \widehat{M}\mathbb E[\theta_0],  \mathbb E[\theta_0]\big\rangle dt\right].
\end{aligned}
\end{equation*}
Let $K >0$ be a constant, which is large enough such that $\max\{|M(t)|, |\widehat{M}(t)|\}\leq K$ for a. e. $t\in[0, T]$. We then further obtain
\begin{equation*}
\begin{aligned}
&\mathcal{J}_{\lambda}(\xi;  u,  v)\\=&\ \frac{1}{2}\mathbb E\left[\lambda\langle X(T),  X(T)\rangle+\int_{0}^T\left(\big\langle M (\theta_0+\theta_1),  \theta_0+\theta_1\big\rangle+\big\langle \widehat{M}\mathbb E[\theta_0+\theta_1],  \mathbb E[\theta_0+\theta_1]\big\rangle\right) dt\right]
\\=&\  J_0(0;  u)+\frac{1}{2}\mathbb E\Big[\lambda\langle X(T),  X(T)\rangle-\langle Gy_0(0),  y_0(0)\rangle-\langle \widehat{G}\mathbb E[y_0(0)],  \mathbb E[y_0(0)]\rangle\Big]
\\&+\frac{1}{2}\mathbb E\left[\int_{0}^T\Big(2\big\langle M\theta_0,  \theta_1\big\rangle+\big\langle M\theta_1,  \theta_1\big\rangle+2\big\langle \widehat{M}\mathbb E[\theta_0],  \mathbb E[\theta_1]\big\rangle+\big\langle \widehat{M}\mathbb E[\theta_1],  \mathbb E[\theta_1]\big\rangle\Big) dt\right]
\\\geq&\ J_0(0;  u)+\frac{1}{2}\mathbb E\Big[\lambda\langle X(T),  X(T)\rangle-\langle Gy_0(0),  y_0(0)\rangle-\langle \widehat{G}\mathbb E[y_0(0)],  \mathbb E[y_0(0)]\rangle\Big]
\\&-K\left\{(\mu+1)\mathbb E\left[\int_{0}^T|\theta_1|^2dt\right]+\frac{1}{\mu}\mathbb E\left[\int_{0}^T|\theta_0|^2dt\right]\right\},
\end{aligned}
\end{equation*}
where  $\mu>0$ is a constant to be determined later.
According to Theorem 2.1 in Li et. al. \cite{LSX2019},
$$\mathbb E\left[\int_{0}^T|\theta_0|^2dt\right]\leq K\mathbb E\left[\int_{0}^T|u|^2dt\right], \ \  \mathbb E\left[\int_{0}^T|\theta_1|^2dt\right]\leq K\mathbb E\left[|\zeta|^2\right]=K\mathbb E\left[|X(T)|^2\right].$$
Thus, we have
\begin{equation*}
\begin{aligned}
\mathbb E\left[\int_{0}^T|v|^2dt\right]&\leq\mathbb E\left[\int_{0}^T2(|z_0|^2+|z_1|^2)dt\right]
\\&\leq 2\mathbb E\left[\int_{0}^T|\theta_0|^2dt\right]+2\mathbb E\left[\int_{0}^T|\theta_1|^2dt\right]
\\&\leq 2K\mathbb E\left[\int_{0}^T|u|^2dt\right]+2K\mathbb E\left[|X(T)|^2\right].
\end{aligned}
\end{equation*}
Moreover, if $\xi=0$, we further have
$$\mathbb E\Big[\langle Gy_0(0),  y_0(0)\rangle\Big]=\mathbb E [\langle Gy_1(0),  y_1(0)\rangle\Big]\leq K \mathbb E\left[|X(T)|^2\right], $$
and
$$\langle \widehat{G}\mathbb E[y_0(0)],  \mathbb E[y_0(0)]\rangle=\langle \widehat{G}\mathbb E[y_1(0)],  \mathbb E[y_1(0)]\rangle\leq K\mathbb E\left[|X(T)|^2\right],$$
by Theorem 2.1 in Li et. al. \cite{LSX2019}.
Combining the above equations and letting $\mu=\frac{2K^2}{\alpha}, \lambda\geq\lambda_0=\frac{\alpha}{2}+2K^2(\mu+1)+2K$, we derive
\begin{equation*}
\begin{aligned}
\mathcal{J}_{\lambda}(\xi;  u,  v)\geq&
\frac{1}{2}\left[\lambda-2K^2(\mu+1)\right]\mathbb E\left[|X(T)|^2\right]+\left(\alpha-\frac{K^2}{\mu}\right)\mathbb E\left[\int_{0}^T|u|^2dt\right]
\\&-\frac{1}{2}\mathbb E\Big[\langle Gy_0(0),  y_0(0)\rangle+\langle \widehat{G}\mathbb E[y_0(0)],  \mathbb E[y_0(0)]\rangle\Big]
\\\geq&\frac{\alpha}{4}\mathbb E\left[|X(T)|^2\right]+\frac{\alpha}{2}\mathbb E\left[\int_{0}^T|u|^2dt\right]+K\mathbb E\left[|X(T)|^2\right]
\\&-\frac{1}{2}\mathbb E\Big[\langle Gy_0(0),  y_0(0)\rangle+\langle \widehat{G}\mathbb E[y_0(0)],  \mathbb E[y_0(0)]\rangle\Big]
\\\geq&\frac{\alpha}{8K}\mathbb E\left[\int_{0}^T|v|^2dt\right]+\frac{\alpha}{4}\mathbb E\left[\int_{0}^T|u|^2dt\right]+K\mathbb E\left[|X(T)|^2\right]
\\&-\frac{1}{2}\mathbb E\Big[\langle Gy_0(0),  y_0(0)\rangle+\langle \widehat{G}\mathbb E[y_0(0)],  \mathbb E[y_0(0)]\rangle\Big].
\end{aligned}
\end{equation*}
If  $\xi=0$, it is obvious that
\begin{equation*}
\begin{aligned}
\mathcal{J}_{\lambda}(0;  u,  v)\geq &\frac{\alpha}{8K}\mathbb E\left[\int_{0}^T|v|^2dt\right]+\frac{\alpha}{4}\mathbb E\left[\int_{0}^T|u|^2dt\right].
\end{aligned}
\end{equation*}
Further, if $-G\geq 0$,  $-\widetilde G\geq 0$,
\begin{equation*}
\begin{aligned}
\mathcal{J}_{\lambda}(\xi;  u,  v)\geq&\frac{\alpha}{8K}\mathbb E\left[\int_{0}^T|v|^2dt\right]+\frac{\alpha}{4}\mathbb E\left[\int_{0}^T|u|^2dt\right]+K\mathbb E\left[|X(T)|^2\right].
\end{aligned}
\end{equation*}
The proof is completed.
\end{proof}
\begin{corollary}\label{coro: 1}
Let Assumptions $A1$-$A3$ hold. Problem (FLQ$_\lambda$) is uniquely solvable for $\lambda\geq \lambda_0$. If, in addition, $-G\geq 0$,  $-\widetilde G\geq 0$, then for $\lambda\geq\lambda_0$,
$$\mathcal{V}_{\lambda}(\xi)\geq0.$$
\end{corollary}
\begin{corollary}\label{coro: 2}
Let Assumptions $A1$-$A3$ hold. Then for any $\lambda\geq \lambda_0$,  Riccati equations
\begin{equation}\label{eq: Pi riccati1}
\left\{\begin{aligned}&\dot{\Pi}_{\lambda}+\Pi_{\lambda}A+A^\top \Pi_{\lambda}+Q\\&\ \ -\left(\begin{array}{c}
C^\top\Pi_{\lambda}+S_1^\top\\B^\top\Pi_{\lambda}+S_2^\top\end{array}\right)^\top\left(\begin{array}{cc}
R_{11}+\Pi_{\lambda}&R_{12}\\R_{12}^\top&R_{22}\end{array}\right)^{-1}\left(\begin{array}{c}
C^\top \Pi_{\lambda}+S_1^\top\\B^\top \Pi_{\lambda}+S_2^\top\end{array}\right)=0,\\&\Pi_{\lambda}(T)=\lambda I,
\end{aligned}\right.
\end{equation}
and
\begin{equation}\label{eq: Pi riccati2}
\left\{\begin{aligned}&\dot{\widetilde{\Pi}}_{\lambda}+\widetilde{\Pi}_{\lambda}\widetilde{A}+\widetilde{A}^\top \widetilde{\Pi}_{\lambda}+\widetilde{Q}\\&\ \ -\left(\begin{array}{c}
\widetilde{C}^\top\widetilde{\Pi}_{\lambda}+\widetilde{S}_1^\top\\\widetilde{B}^\top \widetilde{\Pi}_{\lambda}+\widetilde{S}_2^\top\end{array}\right)^\top\left(\begin{array}{cc}
\widetilde{R}_{11}+\Pi_{\lambda}&\widetilde{R}_{12}\\\widetilde{R}_{12}^\top&\widetilde{R}_{22}\end{array}\right)^{-1}\left(\begin{array}{c}
\widetilde{C}^\top \widetilde{\Pi}_{\lambda}+\widetilde{S}_1^\top\\\widetilde{B}^\top \widetilde{\Pi}_{\lambda}+\widetilde{S}_2^\top\end{array}\right)=0, \\&\widetilde{\Pi}_{\lambda}(T)=\lambda I
\end{aligned}
\right.
\end{equation}
admit unique solutions $\Pi_{\lambda},  \widetilde{\Pi}_{\lambda}\in \mathcal L^\infty(0, T;\mathbb S^n)$, respectively, such that
\begin{equation}\label{eq. weight1}
\left(\begin{array}{cc}
R_{11}+\Pi_{\lambda}&R_{12}\\R_{12}^\top&R_{22}\end{array}\right)\gg 0, \  \left(\begin{array}{cc}
\widetilde{R}_{11}+\Pi_{\lambda}&\widetilde{R}_{12}\\\widetilde{R}_{12}^\top&\widetilde{R}_{22}\end{array}\right)\gg 0.\end{equation}
Further, we have  $$\mathcal{V}_{\lambda}(\xi)=\frac{1}{2}\mathbb E\Big[\langle\Pi_{\lambda}(0)(\xi-\mathbb E[\xi]),  \xi-\mathbb E[\xi]\rangle+\langle \widetilde{\Pi}_{\lambda}(0)\mathbb E[\xi],  \mathbb E[\xi]\rangle\Big].
$$
\end{corollary}
\begin{remark}
From \eqref{eq. weight1} and Schur Lemma \cite{BGFB1994}, we can see that $R_{22}\gg 0,  \widetilde R_{22}\gg0$ under Assumption $A3$. These  features are  quite different from mean-field LQ problems of forward stochastic systems, where the uniform positive definiteness of control weighting matrices are neither necessary nor  sufficient for  the uniform convexity of cost functional (see  \cite{LSY2016,sun2017}).
\end{remark}
\subsection{Reductions of Problem (BLQ)}

Based on the above arguments, we make some reductions of Problem (BLQ).  For this end, we introduce a controlled system
\begin{equation*}
\left\{\begin{aligned}
dY_0=&\ \left(AY_0+\widehat{A}\mathbb E[Y_0]+Bu_0+\widehat{B}\mathbb E[u_0]+\mathcal{C}Z_0+\widehat{\mathcal{C}}\mathbb E[Z_0]+f\right)dt+Z_0dW,
\\Y_0(T)=&\ \zeta,  
\end{aligned}
\right.
\end{equation*}
and a cost functional
\begin{equation*}
\begin{aligned}
\widetilde{J}(\zeta;  u_0)=&\frac{1}{2}\mathbb E\left[ 2\langle g,  Y_0(0)\rangle+2\int_{0}^T\Bigg\langle\left(\begin{array}{c}\widetilde{q}\\\widetilde{\rho}_1\\\rho_2\end{array}\right),  \left(\begin{array}{c}Y_0\\Z_0\\u_0\end{array}\right)\Bigg\rangle dt
\right.\\&+\int_{0}^T\Bigg\langle\left(\begin{array}{ccc}
0&\mathcal{S}_{1}&\mathcal{S}_{2}
\\\mathcal{S}_{1}^\top&\mathcal{R}_{11}&0
\\\mathcal{S}_{2}^\top&0&R_{22} \end{array}\right)\left(\begin{array}{c}Y_0\\Z_0\\u_0\end{array}\right),  \left(\begin{array}{c}Y_0\\Z_0\\u_0\end{array}\right)\Bigg\rangle dt
\\&\left.+\int_{0}^T\Bigg\langle\left(\begin{array}{ccc}
0&\widehat{\mathcal{S}}_{1}&\widehat{\mathcal{S}}_{2}
\\\widehat{\mathcal{S}}_{1}^\top&\widehat{\mathcal{R}}_{11}&0
\\\widehat{\mathcal{S}}_{2}^\top&0&\widehat{R}_{22} \end{array}\right)\left(\begin{array}{c}\mathbb E[Y_0]\\\mathbb E[Z_0]\\\mathbb E[u_0]\end{array}\right),  \left(\begin{array}{c}\mathbb E[Y_0]\\\mathbb E[Z_0]\\\mathbb E[u_0]\end{array}\right)\Bigg\rangle dt\right],
\end{aligned}
\end{equation*}
where
\begin{equation*}
\left\{\begin{aligned}
&\widetilde{q}=q+\Phi f+(\widetilde{\Phi}-\Phi) \mathbb E[f], \ \widetilde{\rho}_1=\rho_1-R_{12}R_{22}^{-1}\rho_2-(\widetilde{R}_{12}\widetilde{R}_{22}^{-1}-R_{12}R_{22}^{-1})\mathbb E[\rho_2],\\
&\mathcal C=C-BR_{22}^{-1}R_{12}^\top, \ \widetilde{\mathcal C}=\mathcal C+ \widehat {\mathcal C} =\widetilde{C}-\widetilde{B}\widetilde{R}_{22}^{-1}\widetilde{R}_{12}^\top,
\\&\mathcal S_1=S_1-S_2R_{22}^{-1}R_{12}^\top+\Phi\mathcal{C},
\ \widetilde{\mathcal S}_1=\mathcal S_1+\widehat{\mathcal S}_1=\widetilde{S}_1-\widetilde{S}_2\widetilde{R}_{22}^{-1}\widetilde{R}_{12}^\top
+\widetilde{\Phi}\widetilde{\mathcal{C}},
\\& \mathcal S_2=S_2+\Phi B, \ \widetilde{\mathcal S}_2=\mathcal S_2+\widehat{\mathcal S}_2=\widetilde{S}_2+\widetilde{\Phi} \widetilde{B},
\\& \mathcal R_{11}=R_{11}-R_{12}R_{22}^{-1}R_{12}^\top+\Phi,
\ \widetilde{\mathcal R}_{11}=\mathcal R_{11}+\widehat{\mathcal R}_{11}=\widetilde{R}_{11}-\widetilde{R}_{12}\widetilde{R}_{22}^{-1}\widetilde{R}_{12}^\top+\Phi,
\end{aligned}\right.
\end{equation*}
with $\Phi$ and $\widetilde{\Phi}$ being the solutions of
\begin{equation*}
\left\{\begin{aligned}
&\dot{\Phi}+\Phi A+A^\top \Phi+Q=0
\\&\Phi(0)=-G,
\end{aligned}\right.
\end{equation*}
and
\begin{equation*}
\left\{\begin{aligned}
&\dot{\widetilde{\Phi}}+\widetilde{\Phi} \widetilde{A}+\widetilde{A}^\top \widetilde{\Phi}+\widetilde{Q}=0
\\&\widetilde{\Phi}(0)=-\widetilde{G},
\end{aligned}\right.
\end{equation*}
respectively.
The corresponding  stochastic optimal control problem is  stated as follows.
\\\textbf{Problem (NC-BLQ)}. Find a $u_0^*\in \mathcal U[0,T]$ such that
\begin{equation*}
\widetilde{V}(\zeta)=\widetilde{J}(\zeta;  u_0^*)=\inf_{u_0\in \mathcal U[0, T]}\widetilde{J}(\zeta;  u_0).
\end{equation*}
Here, ``NC" implies that the  cross-product term  of   $u_0$ and $Z_0$  does not appear in  $\widetilde{J}(\zeta;  u_0)$.
\begin{theorem}\label{control system lemma}
 Let Assumptions $A1$-$A3$ hold. For any two pairs $(Y, Z,u)$ and $(Y_0,Z_0, u_0)$, we introduce a linear transformation
\begin{equation}\label{linear   trans1}
\left(\begin{array}{c}
Y_0-\mathbb E [Y_0]\\
Z_0-\mathbb E [Z_0]\\u_0-\mathbb E [u_0]\end{array}\right)=\left(\begin{array}{ccc}
I&0&0\\
0&I&0\\
0&R_{22}^{-1}R_{12}^\top&I\end{array}\right)\left(\begin{array}{c}
 Y-\mathbb E [Y]\\
Z-\mathbb E [Z]\\u-\mathbb E [u]\end{array}\right),
\end{equation}
 \begin{equation}\label{linear   trans2}
\left(\begin{array}{c}
\mathbb E [Y_0]\\
\mathbb E[Z_0]\\\mathbb E[u_0]\end{array}\right)=\left(\begin{array}{ccc}
I&0&0\\
0&I&0\\
0&\widetilde{R}_{22}^{-1}\widetilde{R}_{12}^\top&I\end{array}\right)\left(\begin{array}{c}
\mathbb E [Y]\\
\mathbb E [Z]\\\mathbb E [u]\end{array}\right).
\end{equation}
Then the following two statements are equivalent:
\begin{itemize}
\item[(1):] $(Y, Z,u)$ is an admissible (optimal) triple of Problem (BLQ).
\item[(2):] $(Y_0, Z_0, u_0)$ is an admissible (optimal) triple of Problem (NC-BLQ).
\end{itemize}
Moreover, we have  $$J(\zeta;  u)=\widetilde{J}(\zeta;  u_0)-\frac{1}{2}\mathbb E\left[\langle\Phi(T)(\zeta-\mathbb E[\zeta]),  \zeta-\mathbb E[\zeta]\rangle +\langle\widetilde{\Phi}(T)\mathbb E[\zeta], \mathbb E[\zeta]\rangle\right].$$
\end{theorem}
\begin{proof}
Applying It\^o formula to $\langle\Phi(Y-\mathbb E[Y]), Y-\mathbb E[Y]\rangle+\langle\widetilde{\Phi}\mathbb E[Y],\mathbb E[Y]\rangle$ on time interval $[0, T]$, we have
\begin{equation*}
\begin{aligned}
&\mathbb E\left[\langle\Phi(T)(\zeta-\mathbb E[\zeta]), \zeta-\mathbb E[\zeta]\rangle +\langle\widetilde{\Phi}(T)\mathbb E[\zeta],\mathbb E[\zeta]\rangle\right]
\\&+\mathbb E\left[\langle G(Y(0)-\mathbb E[Y(0)]), Y(0)-\mathbb E[Y(0)]\rangle
+\langle \widetilde{G}\mathbb E[Y(0)], \mathbb E[Y(0)]\rangle\right]
\\=&\mathbb E\left[\int_{0}^T\Bigg\langle\left(\begin{array}{ccc}
-Q&\Phi C&\Phi B
\\C^\top\Phi&\Phi&0
\\B^\top\Phi&0&0 \end{array}\right)\left(\begin{array}{c}Y-\mathbb E[Y]\\Z-\mathbb E[Z]\\u-\mathbb E[u]\end{array}\right), \left(\begin{array}{c}Y-\mathbb E[Y]\\Z-\mathbb E[Z]\\u-\mathbb E[u]\end{array}\right)\Bigg\rangle dt\right.
\\&\left.+\int_{0}^T\Bigg\langle\left(\begin{array}{ccc}
-\widetilde{Q}&\widetilde{\Phi} \widetilde{C}&\widetilde{\Phi} \widetilde{B}
\\\widetilde{C}^\top\widetilde{\Phi} &\Phi&0
\\\widetilde{B}^\top\widetilde{\Phi}&0&0 \end{array}\right)\left(\begin{array}{c}\mathbb E[Y]\\\mathbb E[Z]\\\mathbb E[u]\end{array}\right), \left(\begin{array}{c}\mathbb E[Y]\\\mathbb E[Z]\\\mathbb E[u]\end{array}\right)\Bigg\rangle dt\right.
\\&\left.+\int_{0}^T\Big(2\langle\Phi (Y-\mathbb E[Y]),f-\mathbb E[f]\rangle+2\langle\widetilde{\Phi}\mathbb E[Y], \mathbb E[f]\rangle\Big)\right].
\end{aligned}
\end{equation*}
It follows that
\begin{equation}\label{eq:equi1}
\begin{aligned}
J(\zeta;  u)=\widehat{J}(\zeta;  u)-\frac{1}{2}\mathbb E\left[\langle\Phi(T)(\zeta-\mathbb E[\zeta]),  \zeta-\mathbb E[\zeta]\rangle +\langle\widetilde{\Phi}(T)\mathbb E[\zeta], \mathbb E[\zeta]\rangle\right].
\end{aligned}
\end{equation}
where
\begin{equation*}
\begin{aligned}
&\widehat{J}(\zeta;  u)\\=&\frac{1}{2}\mathbb E\left[\int_{0}^T\Bigg\langle\left(\begin{array}{ccc}
0&S_{1}+\Phi C&S_{2}+\Phi B
\\(S_{1}+\Phi C)^\top&R_{11}+\Phi&R_{12}
\\(S_{2}+\Phi B)^\top&R_{12}^\top&R_{22} \end{array}\right)\left(\begin{array}{c}Y-\mathbb E[Y]\\Z-\mathbb E[Z]\\u-\mathbb E[u]\end{array}\right),  \left(\begin{array}{c}Y-\mathbb E[Y]\\Z-\mathbb E[Z]\\u-\mathbb E[u]\end{array}\right)\Bigg\rangle dt\right.
\\&+\int_{0}^T\Bigg\langle\left(\begin{array}{ccc}
0&\widetilde{S}_{1}+\widetilde{\Phi} \widetilde{C}&\widetilde{S}_{2}+\widetilde{\Phi} \widetilde{B}
\\(\widetilde{S}_{1}+\widetilde{\Phi} \widetilde{C})^\top&\widetilde{R}_{11}+\Phi&\widetilde{R}_{12}
\\(\widetilde{S}_{2}+\widetilde{\Phi} \widetilde{B})^\top&\widetilde{R}_{12}^\top&\widetilde{R}_{22} \end{array}\right)\left(\begin{array}{c}\mathbb E[Y]\\\mathbb E[Z]\\\mathbb E[u]\end{array}\right),  \left(\begin{array}{c}\mathbb E[Y]\\\mathbb E[Z]\\\mathbb E[u]\end{array}\right)\Bigg\rangle dt
\\&\left.+2\int_{0}^T\Bigg\langle\left(\begin{array}{c}\widetilde{q}\\\rho_1\\\rho_2\end{array}\right),  \left(\begin{array}{c}Y\\Z\\u\end{array}\right)\Bigg\rangle dt
+2\langle g,  Y(0)\rangle\right].
\end{aligned}
\end{equation*}
We pose an LQ problem as follows.\\
\textbf{Problem (BLQA)}: Find a $u^*\in \mathcal U[0,  T]$, such that
$$
\widehat{J}(\zeta;  u^*)=\inf_{u\in \mathcal U[0,  T]}\widehat{J}(\zeta;  u).
$$
We observe from  \eqref{eq:equi1} that  $J(\zeta;  u)$ and $\widehat{J}(\zeta;  u)$ differ by only a constant. It is obvious  that $(Y, Z, u)$ is an admissible (optimal) triple of Problem (BLQA) if and only if $(Y, Z, u)$   is an admissible (optimal) triple of Problem (BLQ). In the following, we investigate the relationship between Problem  (BLQA) and Problem (NC-BLQ).
From \eqref{linear   trans1} and \eqref{linear   trans2}, we have
 \begin{equation*}
\left(\begin{array}{c}
 Y-\mathbb E [Y]\\
Z-\mathbb E [Z]\\u-\mathbb E [u]\end{array}\right)=\left(\begin{array}{ccc}
I&0&0\\
0&I&0\\
0&-R_{22}^{-1}R_{12}^\top&I\end{array}\right)\left(\begin{array}{c}
Y_0-\mathbb E [Y_0]\\
Z_0-\mathbb E [Z_0]\\u_0-\mathbb E [u_0]\end{array}\right),
\end{equation*}
 \begin{equation*}
\left(\begin{array}{c}
\mathbb E [Y]\\
\mathbb E [Z]\\\mathbb E [u]\end{array}\right)=\left(\begin{array}{ccc}
I&0&0\\
0&I&0\\
0&-\widetilde{R}_{22}^{-1}\widetilde{R}_{12}^\top&I\end{array}\right)\left(\begin{array}{c}
\mathbb E [Y_0]\\
\mathbb E[Z_0]\\\mathbb E[u_0]\end{array}\right).
\end{equation*}
Linear transformation \eqref{linear   trans1} with \eqref{linear   trans2} is  invertible. Through direct calculations, we can  verify that $(Y, Z, u)$ is an admissible (optimal) triple of Problem (BLQA) if and only if $(Y_0, Z_0, u_0)$   is an admissible (optimal) triple of Problem (NC-BLQ). Moreover,  $\widetilde{J}(\zeta;  u_0)=\widehat{J}(\zeta;  u)$. The others  follow immediately.
\end{proof}

According to Theorem \ref{control system lemma}, we may assume
\begin{equation}\label{eq: simcoeff}
G=0, \widetilde{G}=0,   Q=0,  \widetilde{Q}=0,  R_{12}=0,  \widetilde{R}_{12}=0,
\end{equation}
in cost functional \eqref{cost functional} without lose of generality throughout this paper.
In the case that \eqref{eq: simcoeff} holds,   Riccati equations \eqref{eq: Pi riccati1} and \eqref{eq: Pi riccati2} take
\begin{equation} \label{eq: Pi riccati1 simp}
\left\{\begin{aligned}&\dot{\Pi}_{\lambda}+\Pi_{\lambda}A+A^\top \Pi_{\lambda}\\&\ \ -\left(\begin{array}{c}
C^\top\Pi_{\lambda}+S_1^\top\\B^\top\Pi_{\lambda}+S_2^\top\end{array}\right)^\top\left(\begin{array}{cc}
R_{11}+\Pi_{\lambda}&0\\0&R_{22}\end{array}\right)^{-1}\left(\begin{array}{c}
C^\top \Pi_{\lambda}+S_1^\top\\B^\top \Pi_{\lambda}+S_2^\top\end{array}\right)=0,\\&\Pi_{\lambda}(T)=\lambda I,
\end{aligned}\right.
\end{equation}
and
\begin{equation} \label{eq: Pi riccati2 simp}
\left\{\begin{aligned}&\dot{\widetilde{\Pi}}_{\lambda}+\widetilde{\Pi}_{\lambda}\widetilde{A}+\widetilde{A}^\top \widetilde{\Pi}_{\lambda}\\&\ \ -\left(\begin{array}{c}
\widetilde{C}^\top\widetilde{\Pi}_{\lambda}+\widetilde{S}_1^\top\\\widetilde{B}^\top \widetilde{\Pi}_{\lambda}+\widetilde{S}_2^\top\end{array}\right)^\top\left(\begin{array}{cc}
\widetilde{R}_{11}+\Pi_{\lambda}&0\\0&\widetilde{R}_{22}\end{array}\right)^{-1}\left(\begin{array}{c}
\widetilde{C}^\top \widetilde{\Pi}_{\lambda}+\widetilde{S}_1^\top\\\widetilde{B}^\top \widetilde{\Pi}_{\lambda}+\widetilde{S}_2^\top\end{array}\right)=0, \\&\widetilde{\Pi}_{\lambda}(T)=\lambda I,
\end{aligned}
\right.
\end{equation}
respectively.
\begin{proposition}\label{Th: monoticity}
Let Assumptions $A1$-$A3$ and \eqref{eq: simcoeff} hold. For $\lambda\geq\lambda_0$, let $\Pi_{\lambda},\widehat{\Pi}_{\lambda}$ be the unique solutions of \eqref{eq: Pi riccati1 simp} and \eqref{eq: Pi riccati2 simp}, respectively. Then we have $$\Pi_{\lambda}(t)\geq 0,\widehat{\Pi}_{\lambda}(t)\geq0, \forall t\in [0, T].$$
Moreover, for $\lambda_2>\lambda_1\geq\lambda_0$, we have
$$\Pi_{\lambda_2}(t)> \Pi_{\lambda_1}(t), \widehat{\Pi}_{\lambda_2}(t)> \widehat{\Pi}_{\lambda_1}(t), \forall t\in [0, T]. $$
\end{proposition}
\begin{proof}
For Problem (FLQ$_{\lambda}$) with $\lambda\geq\lambda_0$,  Theorem \ref{th: connection} and Corollary \ref{coro: 1} give
$$\mathcal{V}_{\lambda}(\xi)=\frac{1}{2}\mathbb E\Big[\langle\Pi_{\lambda}(0)(\xi-\mathbb E[\xi]),  \xi-\mathbb E[\xi]\rangle+\langle \widetilde{\Pi}_{\lambda}(0)\mathbb E[\xi]),  \mathbb E[\xi]\rangle\Big]\geq0.$$
For any  $\xi\neq0$ with $\mathbb{E}[\xi]=0$,
$$
\mathcal{V}_{\lambda}(\xi)=\frac{1}{2}\mathbb E\Big[\langle\Pi_{\lambda}(0)(\xi-\mathbb E[\xi]),  \xi-\mathbb E[\xi]\rangle\Big]\geq0,
$$
which implies that $\Pi_{\lambda}(0)\geq 0$.
For any $\xi\neq0$ with $\xi=\mathbb{E}[\xi]$,
$$
\mathcal{V}_{\lambda}(\xi)=\frac{1}{2}\mathbb E\Big[\langle\widetilde{\Pi}_{\lambda}(0)\mathbb E[\xi], \mathbb E[\xi]\rangle\Big]\geq0,
$$
which implies that $\widetilde{\Pi}_{\lambda}(0)\geq 0$.
Denote
\begin{equation*}
\mathcal Q_{\lambda}=\left(\begin{array}{c}
C^\top\Pi_{\lambda}+S_1^\top\\B^\top\Pi_{\lambda}+S_2^\top\end{array}\right)^\top\left(\begin{array}{cc}
R_{11}+\Pi_{\lambda}&0\\0&R_{22}\end{array}\right)^{-1}\left(\begin{array}{c}
C^\top \Pi_{\lambda}+S_1^\top\\B^\top \Pi_{\lambda}+S_2^\top\end{array}\right),
\end{equation*}
and
\begin{equation*}
\widetilde{\mathcal Q}_{\lambda}=\left(\begin{array}{c}
\widetilde{C}^\top\widetilde{\Pi}_{\lambda}+\widetilde{S}_1^\top\\\widetilde{B}^\top \widetilde{\Pi}_{\lambda}+\widetilde{S}_2^\top\end{array}\right)^\top\left(\begin{array}{cc}
\widetilde{R}_{11}+\Pi_{\lambda}&0\\0&\widetilde{R}_{22}\end{array}\right)^{-1}\left(\begin{array}{c}
\widetilde{C}^\top \widetilde{\Pi}_{\lambda}+\widetilde{S}_1^\top\\\widetilde{B}^\top \widetilde{\Pi}_{\lambda}+\widetilde{S}_2^\top\end{array}\right).
\end{equation*}
Moreover,  let $\Phi$  and $\widetilde{\Phi} $ be     solutions of  matrix ODEs
\begin{equation*}
\left\{\begin{aligned}
&d\Phi=A\Phi dt,
\\&\Phi(0)= I,
\end{aligned}
\right.
\end{equation*}
and
\begin{equation*}
\left\{\begin{aligned}
&d\widetilde{\Phi}=\widetilde{A}\widetilde{\Phi} dt,
\\&\widetilde{\Phi}(0)= I,
\end{aligned}
\right.
\end{equation*}
respectively. Equations \eqref{eq: simcoeff} and \eqref{eq: Pi riccati1 simp} imply that
$$\Pi_\lambda(t)=\left[\Phi(t)^{-1}\right]^\top\left[\Pi_\lambda(0)+\int_0^t\Phi(s)^\top\mathcal Q_{\lambda}(s)\Phi(s)ds\right]\Phi(t)^{-1},$$
and
$$\widetilde{\Pi}_\lambda(t)=\left[\widetilde{\Phi}(t)^{-1}\right]^\top\left[\widetilde{\Pi}_\lambda(0)+\int_0^t\widetilde{\Phi}(s)^\top\widetilde{\mathcal Q}_{\lambda}(s)\widetilde{\Phi}(s)ds\right]\widetilde{\Phi}(t)^{-1}.$$
From Corollary \ref{coro: 2}, it is easy to  see that $$\Pi_{\lambda}(t)\geq 0, \widetilde{\Pi}_{\lambda}(t)\geq 0,  \forall t\in [0,  T].$$

Denote $\mathcal{B}=(C,  B),  \widehat{\mathcal{B}}=(\widehat{C},  \widehat{B}),  \widetilde{\mathcal{B}}=\mathcal{B}+\widehat{\mathcal{B}},  \mathcal{D}=(I,  0)$. With these notations, Riccati equations \eqref{eq: Pi riccati1 simp} and \eqref{eq: Pi riccati2 simp} can be rewritten as
\begin{equation*}
\left\{\begin{aligned}&\dot{\Pi}_{\lambda}+\Pi_{\lambda}A+A^\top \Pi_{\lambda}-\left(\Pi_{\lambda}\mathcal{B}+S\right)\left(
R+\mathcal{D}^\top \Pi_{\lambda}\mathcal{D}\right)^{-1}\left(\Pi_{\lambda}\mathcal{B}+S\right)^\top=0, \\&\Pi_{\lambda}(T)=\lambda I,
\end{aligned}\right.
\end{equation*}
and
\begin{equation*}
\left\{\begin{aligned}&\dot{\widetilde{\Pi}}_{\lambda}+\widetilde{\Pi}_{\lambda}\widetilde{A}+\widetilde{A}^\top \widetilde{\Pi}_{\lambda}-\left(\widetilde{\Pi}_{\lambda}\widetilde{\mathcal{B}}+\widetilde{S}\right)\left(
\widetilde{R}+\mathcal{D}^\top \Pi_{\lambda}\mathcal{D}\right)^{-1}\left(\widetilde{\Pi}_{\lambda}\widetilde{\mathcal{B}}+\widetilde{S}\right)^\top=0, \\&\widetilde{\Pi}_{\lambda}(T)=\lambda I,
\end{aligned}
\right.
\end{equation*}
respectively.
For $\lambda_2>\lambda_1\geq\lambda_0$, letting $\Delta=\Pi_{\lambda_2}-\Pi_{\lambda_1},  \widetilde{\Delta}=\widetilde{\Pi}_{\lambda_2}-\widetilde{\Pi}_{\lambda_1}$, we obtain
\begin{equation*}
\left\{\begin{aligned}&\dot{\Delta}+\Delta A+A^\top \Delta+\left(\Pi_{\lambda_1}\mathcal{B}+S\right)\left(
R+\mathcal{D}^\top \Pi_{\lambda_1}\mathcal{D}\right)^{-1}\left(\Pi_{\lambda_1}\mathcal{B}+S\right)^\top \\&-\left(\Delta\mathcal{B}+\Pi_{\lambda_1}\mathcal{B}+S\right)\left(
R+\mathcal{D}^\top\Pi_{\lambda_1}\mathcal{D}+\mathcal{D}^\top \Delta\mathcal{D}\right)^{-1}\left(\Delta\mathcal{B}+\Pi_{\lambda_1}\mathcal{B}+S\right)^\top=0,
\\&\Delta(T)=(\lambda_2-\lambda_1) I,
\end{aligned}\right.
\end{equation*}
and
\begin{equation*}
\left\{\begin{aligned}&\dot{\widetilde{\Delta}}+\widetilde{\Delta}\widetilde{A}+\widetilde{A}^\top \widetilde{\Delta}+\left(\widetilde{\Pi}_{\lambda_1}\widetilde{\mathcal{B}}+\widetilde{S}\right)\left(
\widetilde{R}+\mathcal{D}^\top \Pi_{\lambda_1}\mathcal{D}\right)^{-1}\left(\widetilde{\Pi}_{\lambda_1}\widetilde{\mathcal{B}}+S\right)^\top
\\&-\left(\widetilde{\Delta}\widetilde{\mathcal{B}}+\widetilde{\Pi}_{\lambda_1}\widetilde{\mathcal{B}}+\widetilde{S}\right)\left(
\widetilde{R}+\mathcal{D}^\top \Pi_{\lambda_1}\mathcal{D}+\mathcal{D}^\top \Delta\mathcal{D}\right)^{-1}\left(\widetilde{\Delta}\widetilde{\mathcal{B}}+\widetilde{\Pi}_{\lambda_1}\widetilde{\mathcal{B}}
+\widetilde{S}\right)^\top=0, \\&\widetilde{\Delta}(T)=(\lambda_2-\lambda_1) I.
\end{aligned}
\right.
\end{equation*}
According to Corollary  2.3 in Sun et. al. \cite{SWX2022}, we have
$$\Pi_{\lambda_2}(t)> \Pi_{\lambda_1}(t), \widetilde{\Pi}_{\lambda_2}(t)> \widetilde{\Pi}_{\lambda_1}(t), \forall t\in [0, T].$$
\end{proof}
\subsection{Optimal control and value function}
In this subsection,  we construct the optimal control of Problem (BLQ) in the case that  Assumption $A3$ and \eqref{eq: simcoeff} hold via an adjoint process, two Riccati equations with terminal conditions and a mean-field BSDE. For this end,
we introduce the following Ricccati  equations
\begin{equation}\label{eq.RE1}
\left\{
\begin{aligned}
&\dot\Upsilon-\Upsilon A^\top-A\Upsilon
\\&+\left(\begin{array}{cc}C^\top+S_1^\top\Upsilon\\B^\top+S_2^\top\Upsilon \end{array}\right)^\top\left(\begin{array}{cc}I+\Upsilon R_{11}& 0\\0&R_{22}\end{array}\right)^{-1}\left(\begin{array}{c}\Upsilon C^\top+\Upsilon S_1^\top\Upsilon\\B^\top+S_2^\top\Upsilon\end{array}\right)=0, \\
&\Upsilon(T)=0,
\end{aligned}
\right.
\end{equation}
and
\begin{equation}\label{eq.RE2}
\left\{
\begin{aligned}
&\dot{\widetilde{\Upsilon}}-\widetilde{\Upsilon} \widetilde{A}^\top-\widetilde{A}\widetilde{\Upsilon}
\\&+\left(\begin{array}{c}\widetilde{C}^\top+\widetilde{S}_1^\top\widetilde{\Upsilon}
\\\widetilde{B}^\top+\widetilde{S}_2^\top\widetilde{\Upsilon}\end{array}\right)^\top
\left(\begin{array}{cc}I+\Upsilon \widetilde{R}_{11}&0\\0&\widetilde{R}_{22}\end{array}\right)^{-1}\left(\begin{array}{c}\Upsilon \widetilde{C}^\top+\Upsilon  \widetilde{S}_1^\top\widetilde{\Upsilon}
\\\widetilde{B}^\top+\widetilde{S}_2^\top\widetilde{\Upsilon}\end{array}\right)=0, \\
&\widetilde{\Upsilon}(T)=0.
\end{aligned}
\right.
\end{equation}
\begin{theorem}\label{Th: solva RE backward}
Let Assumptions $A1$-$A3$ and \eqref{eq: simcoeff} hold. Then Riccati equations \eqref{eq.RE1} and  \eqref{eq.RE2} admit unique solutions $\Upsilon\geq 0,  \widetilde{\Upsilon}\geq 0$, such that $I+\Upsilon R_{11},  I+\Upsilon \widetilde{R}_{11}$ are invertible on $[0, T]$,  $(I+\Upsilon R_{11})^{-1},  (I+\Upsilon \widetilde{R}_{11})^{-1}\in\mathcal L^\infty(0, T;\mathbb R^{n\times n})$,  $(I+\Upsilon R_{11})^{-1}\Upsilon\geq 0,  (I+\Upsilon \widetilde{R}_{11})^{-1}\Upsilon\geq 0$.
\end{theorem}
\begin{proof}
Uniqueness: Suppose  that Riccati equation  \eqref{eq.RE1} admits two solutions $\Upsilon^1$ and  $\Upsilon^2$.   Letting $\delta\Upsilon=\Upsilon^1-\Upsilon^2$, we then have
\begin{equation*}
\begin{aligned}
&\frac{d}{dt}\delta\Upsilon\\=
&\delta\Upsilon A^\top+A\delta\Upsilon-\delta\Upsilon S_1(I+\Upsilon^1R_{11})^{-1}\Upsilon^1 (C^\top+S_1^\top\Upsilon^1)
\\&-(C+\Upsilon^2 S_1)\left[(I+\Upsilon^1 R_{11})^{-1}\Upsilon^1(C^\top+ S_1^\top\Upsilon^1)-(I+\Upsilon^2 R_{11})^{-1}\Upsilon^2(C^\top+ S_1^\top\Upsilon^2)\right]
\\&-\delta\Upsilon S_2 R_{22}^{-1}(B^\top+S_2^\top\Upsilon^1)-(B+\Upsilon^2 S_2)R_{22}^{-1}S_2^\top \delta\Upsilon
\\=&\delta\Upsilon A^\top+A\delta\Upsilon-\delta\Upsilon S_1(I+\Upsilon^1R_{11})^{-1}\Upsilon^1 (C^\top+S_1^\top\Upsilon^1)
\\&-(C+\Upsilon^2 S_1)\Big\{-(I+\Upsilon^1R_{11})^{-1}\delta\Upsilon R_{11}(I+\Upsilon^2R_{11})^{-1} \Upsilon^1 (C^\top+S_1^\top\Upsilon^1)\\&+(I+\Upsilon^2R_{11})^{-1}\left[\delta\Upsilon
(C^\top+S_1^\top\Upsilon^1)+\Upsilon^2S_1^\top \delta\Upsilon\right]\Big\}
\\&-\delta\Upsilon S_2 R_{22}^{-1}(B^\top+S_2^\top\Upsilon^1)-(B+\Upsilon^2 S_2)R_{22}^{-1} S_2^\top \delta\Upsilon
\\\triangleq& f(t,  \delta\Upsilon).
\end{aligned}
\end{equation*}
Recalling that $\delta\Upsilon(T)=0$ and $f(t,x)$ is Lipschitz with respect to $x$, Gronwall's inequality implies that $\delta\Upsilon=0$ for all $t\in [0, T]$.

Existence:  For $\lambda>\lambda_0$,  let  $\Pi_{\lambda}$ and $\widetilde{\Pi}_{\lambda}$ be unique solutions of \eqref{eq: Pi riccati1 simp} and \eqref{eq: Pi riccati2 simp}, respectively. According to Proposition \ref{Th: monoticity},    $$\Pi_{\lambda}(t)> 0, \widetilde{\Pi}_{\lambda}(t)> 0,  \forall t\in [0,  T].$$
Defining $$\Upsilon_{\lambda}=\Pi_{\lambda}^{-1},  \widetilde{\Upsilon}_{\lambda}=\widetilde{\Pi}_{\lambda}^{-1},  $$
we have
$\Upsilon_{\lambda}(t)\geq 0$ and $\widetilde{\Upsilon}_{\lambda}(t)\geq 0$ are decreasing with respect to $\lambda$. 
Assuming that $\Upsilon_{\lambda}$ and $\widetilde{\Upsilon}_{\lambda}$ converge pointwise to $\Upsilon$ and $\widetilde{\Upsilon}$, respectively,  we obtain $\Upsilon\geq 0$ and $\widetilde{\Upsilon}\geq 0$.

We now proceed to prove the following:
\begin{itemize}
\item[(i).] $I+\Upsilon R_{11}, I+\Upsilon \widetilde{R}_{11}$ are invertible on $[0, T]$;
\item[(ii).] $(I+\Upsilon R_{11})^{-1}\in\mathcal L^\infty(0,T;\mathbb R^{n\times n}), (I+\Upsilon \widetilde{R}_{11})^{-1}\in\mathcal L^\infty(0,T;\mathbb R^{n\times n})$;
\item[(iii).] $(I+\Upsilon R_{11})^{-1}\Upsilon\geq 0,  (I+\Upsilon \widetilde{R}_{11})^{-1}\Upsilon\geq 0$;
\item[(iv).] $\Upsilon$ and $\widetilde{\Upsilon}$ satisfy  equations \eqref{eq.RE1} and  \eqref{eq.RE2}, respectively.
\end{itemize}

For (i)  and (ii), we note that for $\lambda>\lambda_0$,
\begin{equation}\label{eq: F-uniconvex}
\begin{aligned}
&\Pi_{\lambda}(I+\Upsilon_{\lambda} R_{11})=\Pi_{\lambda}+R_{11}\gg 0,
\\&\Pi_{\lambda}(I+\Upsilon_{\lambda} \widetilde{R}_{11})=\Pi_{\lambda}+\widetilde{R}_{11}\gg 0,
\end{aligned}
\end{equation}
which imply that $I+\Upsilon_{\lambda} R_{11}, I+\Upsilon_{\lambda} \widetilde{R}_{11}$ are invertible on $[0, T]$.
By Proposition \ref{Th: monoticity},   for $\lambda>\lambda_0$,
$$\Pi_{\lambda}(t)> \Pi_{\lambda_0}(t), \widetilde{\Pi}_{\lambda}(t)> \widetilde{\Pi}_{\lambda_0}(t), \forall t\in [0, T]. $$
Thus, $$0\leq (R_{11}+\Pi_{\lambda})^{-1}\leq (R_{11}+\Pi_{\lambda_0})^{-1},$$
and
$$0\leq (\widetilde{R}_{11}+\Pi_{\lambda})^{-1}\leq (\widetilde{R}_{11}+\Pi_{\lambda_0})^{-1}.$$
From the above arguments, we obtain for every $x\in  \mathbb R^n$
\begin{equation*}
\begin{aligned}
&\langle\Pi_{\lambda}(R_{11}+\Pi_{\lambda})^{-2}\Pi_{\lambda} x,  x\rangle\\=&|(R_{11}+\Pi_{\lambda_0}+\Pi_{\lambda}-\Pi_{\lambda_0})^{-1}
(\Pi_{\lambda}-\Pi_{\lambda_0}+\Pi_{\lambda_0})x|^2
\\\leq &2|(R_{11}+\Pi_{\lambda_0}+\Pi_{\lambda}-\Pi_{\lambda_0})^{-1}
(\Pi_{\lambda}-\Pi_{\lambda_0})x|^2+2|(R_{11}+\Pi_{\lambda_0}+\Pi_{\lambda}-\Pi_{\lambda_0})^{-1}
\Pi_{\lambda_0}x|^2
\\=&2|x-(R_{11}+\Pi_{\lambda_0}+\Pi_{\lambda}-\Pi_{\lambda_0})^{-1}
(R_{11}+\Pi_{\lambda_0})x|^2+2|(R_{11}+\Pi_{\lambda_0}+\Pi_{\lambda}-\Pi_{\lambda_0})^{-1}
\Pi_{\lambda_0}x|^2
\\\leq &4\left[1+\big|(R_{11}+\Pi_{\lambda_0}+\Pi_{\lambda}-\Pi_{\lambda_0})^{-1}\big|^2
\left(|R_{11}+\Pi_{\lambda_0}|^2+|\Pi_{\lambda_0}|\right)^2\right]|x|^2
\\\leq &4\left[1+\big|(R_{11}+\Pi_{\lambda_0})^{-1}\big|^2
\left(|R_{11}+\Pi_{\lambda_0}|^2+|\Pi_{\lambda_0}|\right)^2\right]|x|^2.
\end{aligned}
\end{equation*}
On the other hand,
\begin{equation*}
\begin{aligned}
(\Pi_{\lambda}+R_{11})^{-1}\Pi_{\lambda}(I+\Upsilon_{\lambda} R_{11})(I+\Upsilon_{\lambda} R_{11})^\top\Pi_{\lambda}(\Pi_{\lambda}+R_{11})^{-1}=I,
\end{aligned}
\end{equation*}
which implies that
\begin{equation*}
\begin{aligned}
&(I+\Upsilon_{\lambda} R_{11})(I+\Upsilon_{\lambda} R_{11})^\top
\\=&\left[\Pi_{\lambda}(\Pi_{\lambda}+R_{11})^{-2}\Pi_{\lambda}\right]^{-1}
\\\geq &\frac{1}{4}\left[1+\big|(R_{11}+\Pi_{\lambda_0})^{-1}\big|^2
\left(|R_{11}+\Pi_{\lambda_0}|^2+|\Pi_{\lambda_0}|\right)^2\right]^{-1}I.
\end{aligned}
\end{equation*}
Letting $\lambda\to +\infty$, we get
$$(I+\Upsilon R_{11})(I+\Upsilon R_{11})^\top\geq \frac{1}{4}\left[1+\big|(R_{11}+\Pi_{\lambda_0})^{-1}\big|^2
\left(|R_{11}+\Pi_{\lambda_0}|^2+|\Pi_{\lambda_0}|\right)^2\right]^{-1}I.$$
Thus,  $I+\Upsilon R_{11}$ is invertible on $[0, T]$ and $(I+\Upsilon R_{11})^{-1}\in\mathcal L^\infty(0,T;\mathbb R^{n\times n})$.  $I+\Upsilon\widetilde{ R}_{11}$ is invertible on $[0, T]$ and $(I+\Upsilon \widetilde{R}_{11})^{-1}\in\mathcal L^\infty(0,T;\mathbb R^{n\times n})$ can be derived similarly.

For (iii),   \eqref{eq: F-uniconvex} implies that
\begin{equation*}
\begin{aligned}
&(I+\Upsilon_{\lambda} R_{11})^{-1}\Upsilon_{\lambda}=(\Pi_{\lambda}+R_{11})^{-1}\geq 0,
\\&(I+\Upsilon_{\lambda} \widetilde{R}_{11})^{-1}\Upsilon_{\lambda}=(\Pi_{\lambda}+\widetilde{R}_{11})^{-1}\geq 0.
\end{aligned}
\end{equation*}
Thus
$(I+\Upsilon_{\lambda} R_{11})^{-1}\Upsilon_{\lambda},  (I+\Upsilon_{\lambda} \widetilde{R}_{11})^{-1}\Upsilon_{\lambda}$ are both decreasing with respect to $\lambda$ and bounded
below by zero. Letting $\lambda\to +\infty$, we obtain
$$(I+\Upsilon R_{11})^{-1}\Upsilon\geq 0,
(I+\Upsilon\widetilde{R}_{11})^{-1}\Upsilon\geq 0.$$

In the following, we prove that $\Upsilon$ and $\widetilde{\Upsilon}$ satisfy  equations \eqref{eq.RE1} and  \eqref{eq.RE2}, respectively.  From the definition of $\Upsilon_{\lambda}$, we get
$$\frac{d}{dt}\Upsilon_{\lambda}(t)\Pi_{\lambda}(t)+\Upsilon_{\lambda}(t)\frac{d}{dt}\Pi_{\lambda}(t)
=\frac{d}{dt}\left(\Upsilon_{\lambda}(t)\Pi_{\lambda}(t)\right)=0,$$
which implies that
\begin{equation*}
\begin{aligned}
&\dot{\Upsilon}_{\lambda}
\\=&\Upsilon_{\lambda}\left[\Pi_{\lambda}A+A^\top \Pi_{\lambda}-\left(\setlength{\arraycolsep}{0.7pt}\begin{array}{c}
C^\top\Pi_{\lambda}+S_1^\top\\B^\top\Pi_{\lambda}+S_2^\top\end{array}\right)^\top\left(\setlength{\arraycolsep}{0.7pt}\begin{array}{cc}
R_{11}+\Pi_{\lambda}&0\\0&R_{22}\end{array}\right)^{-1}\left(\setlength{\arraycolsep}{0.7pt}\begin{array}{c}
C^\top \Pi_{\lambda}+S_1^\top\\B^\top \Pi_{\lambda}+S_2^\top\end{array}\right)\right]\Upsilon_{\lambda}
\\=&A\Upsilon_{\lambda}+\Upsilon_{\lambda} A^\top
-\left(\setlength{\arraycolsep}{0.7pt}\begin{array}{c}
C^\top+S_1^\top\Upsilon_{\lambda}\\B^\top+S_2^\top\Upsilon_{\lambda}\end{array}\right)^\top
\left(\setlength{\arraycolsep}{0.7pt}\begin{array}{cc}
R_{11}+\Pi_{\lambda}&0\\0&R_{22}\end{array}\right)^{-1}\left(\setlength{\arraycolsep}{0.7pt}\begin{array}{c}
C^\top+S_1^\top\Upsilon_{\lambda}\\B^\top+S_2^\top\Upsilon_{\lambda}\end{array}\right)
\\=&A\Upsilon_{\lambda}+\Upsilon_{\lambda} A^\top
-\left(\setlength{\arraycolsep}{0.7pt}\begin{array}{c}
C^\top+S_1^\top\Upsilon_{\lambda}\\B^\top+S_2^\top\Upsilon_{\lambda}\end{array}\right)^\top
\left(\setlength{\arraycolsep}{0.7pt}\begin{array}{cc}(I+\Upsilon_{\lambda} R_{11})^{-1}\Upsilon_{\lambda}& 0\\0&R_{22}^{-1}\end{array}\right)\left(\setlength{\arraycolsep}{0.7pt}\begin{array}{c}
C^\top+S_1^\top\Upsilon_{\lambda}\\B^\top+S_2^\top\Upsilon_{\lambda}\end{array}\right).
\end{aligned}
\end{equation*}
Consequently,
\begin{equation*}
\begin{aligned}
&\Upsilon_{\lambda}(t)\\=&\frac{1}{\lambda}I
-\int_t^T\left[A\Upsilon_{\lambda}+\Upsilon_{\lambda} A^\top
-\left(\begin{array}{c}
C^\top+S_1^\top\Upsilon_{\lambda}\\B^\top+S_2^\top\Upsilon_{\lambda}\end{array}\right)^\top
\left(\begin{array}{cc}(I+\Upsilon_{\lambda} R_{11})^{-1}\Upsilon_{\lambda}& 0\\0&R_{22}^{-1}\end{array}\right)\left(\begin{array}{c}
C^\top+S_1^\top\Upsilon_{\lambda}\\B^\top+S_2^\top\Upsilon_{\lambda}\end{array}\right) \right]ds.
\end{aligned}
\end{equation*}
Letting $\lambda\to +\infty$, then  $\Upsilon(\cdot)$ is given by
\begin{equation*}
\begin{aligned}
\Upsilon(t)=-\int_t^T\left[A\Upsilon+\Upsilon A^\top
-\left(\begin{array}{c}
C^\top+S_1^\top\Upsilon\\B^\top+S_2^\top\Upsilon\end{array}\right)^\top
\left(\begin{array}{cc}(I+\Upsilon R_{11})^{-1}\Upsilon& 0\\0&R_{22}^{-1}\end{array}\right)\left(\begin{array}{c}
C^\top+S_1^\top\Upsilon\\B^\top+S_2^\top\Upsilon\end{array}\right) \right]ds,
\end{aligned}
\end{equation*}
which definitely satisfies Riccati equation \eqref{eq.RE1}. Following a similar procedure as above, we  obtain $\widetilde{\Upsilon}$ satisfies Riccati equation  \eqref{eq.RE2}.
\end{proof}
We now introduce the following mean-field BSDE.
\begin{equation}\label{eq.BSDE1}
\left\{
\begin{aligned}
&d\eta=\Big[A(\eta-\mathbb E [\eta])+\widetilde{A}\mathbb E [\eta]+f+\Upsilon (q-\mathbb E[q])+\widetilde{\Upsilon}\mathbb E[q]\\&\ \ -\left(\begin{array}{cc}C^\top+S_1^\top\Upsilon\\B^\top+S_2^\top\Upsilon \end{array}\right)^\top\left(\begin{array}{cc}I+\Upsilon R_{11}& 0\\0&R_{22}\end{array}\right)^{-1}\left(\begin{array}{c}\Upsilon S_1^\top(\eta-\mathbb E[\eta])+\Upsilon (\rho_1-\mathbb E[\rho_1])-(\beta-\mathbb E[\beta])\\S_2^\top(\eta-\mathbb E[\eta])+(\rho_2-\mathbb E[\rho_2])\end{array}\right)
\\&\ \ -\left(\begin{array}{c}\widetilde{C}^\top+\widetilde{S}_1^\top\widetilde{\Upsilon}
\\\widetilde{B}^\top+\widetilde{S}_2^\top\widetilde{\Upsilon}\end{array}\right)^\top\left(\begin{array}{cc}I+\Upsilon \widetilde{R}_{11}& 0\\0&\widetilde{R}_{22}\end{array}\right)^{-1}\left(\begin{array}{c}\Upsilon \widetilde{S}_1^\top\mathbb E[\eta]+\Upsilon \mathbb E[\rho_1]-\mathbb E [\beta]\\\widetilde{S}_2^\top\mathbb E [\eta]+\mathbb E[\rho_2]\end{array}\right)\Bigg]dt+\beta dW, \\
&\eta(T)=\ \zeta.
\end{aligned}
\right.
\end{equation}
\begin{theorem}\label{Th: no-feedback}
Suppose  Assumptions $A1$-$A3$ and \eqref{eq: simcoeff} hold. Let $\Upsilon$ and $\widetilde{\Upsilon}$ be solutions of Riccati equations \eqref{eq.RE1} and  \eqref{eq.RE2}, respectively. Let $(\eta, \beta)$ be an adapted solution of mean-field BSDE \eqref{eq.BSDE1}. Then the optimal control of Problem (BLQ) is given by
\begin{equation}\label{optimal control}
\begin{aligned}
u=&-R_{22}^{-1}\left[(B^\top +S_2^\top\Upsilon)(X-\mathbb E[X])+S_2^\top(\eta-\mathbb E[\eta])+(\rho_2-\mathbb E[\rho_2]) \right]
\\&-\widetilde{R}_{22}^{-1}\left[(\widetilde{B}^\top +\widetilde{S}_2^\top\widetilde{\Upsilon})\mathbb E[X]+\widetilde{S}_2^\top\mathbb E[\eta]+\mathbb E[\rho_2]\right],
\end{aligned}
\end{equation}
where $X$ is given by
\begin{equation}\label{adjoint decouple}
\left\{\begin{aligned}
&dX=-\bigg\{\left[A^\top-S_1(I+\Upsilon R_{11})^{-1}\Upsilon(C^\top+S_1^\top \Upsilon)-S_2R_{22}^{-1}(B^\top+S_2^\top \Upsilon)\right](X-\mathbb E[X])
\\&\qquad-S_1(I+\Upsilon R_{11})^{-1}\left[\Upsilon S_1^\top (\eta-\mathbb E[\eta])+\Upsilon(\rho_1-\mathbb E[\rho_1])-(\beta-\mathbb E[\beta])\right]\\&\qquad-S_2 R_{22}^{-1}\Big(S_2^\top (\eta-\mathbb E[\eta])+(\rho_2-\mathbb E[\rho_2])\Big)
\\&\qquad+\left[\widetilde{A}^\top-\widetilde{S}_1(I+\Upsilon \widetilde{R}_{11})^{-1}\Upsilon(\widetilde{C}^\top+\widetilde{S}_1^\top \widetilde{\Upsilon})-\widetilde{S}_2\widetilde{R}_{22}^{-1}(\widetilde{B}^\top+\widetilde{S}_2^\top \widetilde{\Upsilon})\right]\mathbb E[X]
\\&\qquad-\widetilde{S}_1(I+\Upsilon \widetilde{R}_{11})^{-1}\left[\Upsilon \widetilde{S}_1^\top \mathbb E[\eta]+\Upsilon\mathbb E[\rho_1]-\mathbb E[\beta]\right]-\widetilde{S}_2 \widetilde{R}_{22}^{-1}\Big(\widetilde{S}_2^\top\mathbb E[\eta]+\mathbb E[\rho_2]\Big)+q\bigg\}dt
\\&\qquad-\bigg\{(I+ R_{11}\Upsilon)^{-1}\left[(C^\top+S_1^\top \Upsilon)(X-\mathbb E[X])
+S_1^\top (\eta-\mathbb E[\eta])
+(\rho_1-\mathbb E[\rho_1])+R_{11}(\beta-\mathbb E[\beta])\right]
\\&\qquad+(I+ \widetilde{R}_{11}\Upsilon)^{-1}\left[(\widetilde{C}^\top+\widetilde{S}_1^\top \widetilde{\Upsilon})\mathbb E[X]
+\widetilde{S}_1^\top \mathbb E[\eta]
+\mathbb E[\rho_1]+\widetilde{R}_{11}\mathbb E[\beta]\right]
\bigg\}dW,
\\&X(0)=-g.  
\end{aligned}
\right.
\end{equation}
Moreover, the value function of Problem (BLQ) is
\begin{equation}\label{eq: optimal cost fuctional}
\begin{aligned}
V(\zeta)=&\frac{1}{2}\mathbb E\left[\langle g,2\eta(0)-\Upsilon(0)(g-\mathbb E[g])-\widetilde{\Upsilon}(0)\mathbb E[g]\rangle\right.
\\&+\int_0^T\Big\{\langle\rho_1,-(I+\Upsilon R_{11})^{-1}\Upsilon(\rho_1-\mathbb E[\rho_1])-(I+\Upsilon \widetilde{R}_{11})^{-1}\Upsilon\mathbb E[\rho_1]\rangle
\\&-\langle\rho_2,R_{22}^{-1}(\rho_2-\mathbb E[\rho_2])+\widetilde{R}_{22}^{-1}\mathbb E[\rho_2]\rangle
+\langle\beta,(I+ R_{11}\Upsilon)^{-1}R_{11}(\beta-\mathbb E[\beta])+(I+ \widetilde{R}_{11}\Upsilon)^{-1}\widetilde{R}_{11}\mathbb E[\beta]\rangle
\\&-\langle\eta,\left(S_1(I+\Upsilon R_{11})^{-1}\Upsilon S_1^\top+S_2R_{22}^{-1}S_2^\top\right)(\eta-\mathbb E[\eta])+\left(\widetilde{S}_1(I+\Upsilon \widetilde{R}_{11})^{-1}\Upsilon \widetilde{S}_1^\top+\widetilde{S}_2\widetilde{R}_{22}^{-1}\widetilde{S}_2^\top\right)\mathbb E[\eta]\rangle
\\&+2\langle \rho_1,(I+\Upsilon R_{11})^{-1}(\beta-\mathbb E[\beta])+(I+\Upsilon \widetilde{R}_{11})^{-1}\mathbb E[\beta]\rangle
\\&-2\langle\eta,S_1(I+\Upsilon R_{11})^{-1}[\Upsilon(\rho_1-\mathbb E[\rho_1])-(\beta-\mathbb E[\beta])]+S_2R_{22}^{-1}(\rho_2-\mathbb E[\rho_2])\rangle
\\&\left.-2\langle\eta,\widetilde{S}_1(I+\Upsilon \widetilde{R}_{11})^{-1}[\Upsilon\mathbb E[\rho_1]-\mathbb E[\beta]]+\widetilde{S}_2\widetilde{R}_{22}^{-1}\mathbb E[\rho_2]-q\rangle\Big\}dt\right].
\end{aligned}
\end{equation}
\end{theorem}
\begin{proof}
Define
\begin{equation*}
\begin{aligned}
Y=&\ \Upsilon(X-\mathbb E[X])+\widetilde{\Upsilon}\mathbb E[X]+\eta,
\\Z=&-(I+\Upsilon R_{11})^{-1}\left[\Upsilon(C^\top +S_1^\top\Upsilon)(X-\mathbb E[X])+\Upsilon S_1^\top(\eta-\mathbb E[\eta])+\Upsilon(\rho_1-\mathbb E[\rho_1])-(\beta-\mathbb E[\beta])\right]
\\&-(I+\Upsilon \widetilde{R}_{11})^{-1}\left[\Upsilon(\widetilde{C}^\top +\widetilde{S}_1^\top\widetilde{\Upsilon})\mathbb E[X]+\Upsilon \widetilde{S}_1^\top\mathbb E[\eta]+\Upsilon\mathbb E[\rho_1]-\mathbb E[\beta]\right].
\end{aligned}
\end{equation*}
With these notations, \eqref{adjoint decouple} can be rewritten as
\begin{equation*}
\left\{\begin{aligned}
&dX=-\left[A^\top X+\widehat{A}^\top \mathbb E[X]+S_1Z+\widehat{S}_1 \mathbb E[Z]+S_2u+\widehat{S}_2\mathbb E[u]+q\right]dt
\\&\qquad\ \ -\left[C^\top X+\widehat{C}^\top \mathbb E[X]+S_1^\top Y+\widehat{S}_1^\top \mathbb E[Y]+R_{11} Z+\widehat{R}_{11}\mathbb E[Z]+\rho_1\right]dW,
\\&X(0)=-g.
\end{aligned}
\right.
\end{equation*}
Further, by It\^o formula, we have
\begin{equation*}
\begin{aligned}
dY=&\dot{\Upsilon}(X-\mathbb E[X])dt
-\Upsilon\left[A^\top(X-\mathbb E[X])+S_1(Z-\mathbb E[Z])+S_2(u-\mathbb E[u])+(q-\mathbb E[q])\right]dt
\\&-\Upsilon\left[C^\top X+\widehat{C}^\top \mathbb E[X]+S_1^\top Y+\widehat{S}_1^\top \mathbb E[Y]+R_{11} Z+\widehat{R}_{11}\mathbb E[Z]+\rho_1\right]dW
\\&+\dot{\widetilde{\Upsilon}}\mathbb E[X]dt-\widetilde{\Upsilon}\left[\widetilde{A}^\top \mathbb E[X]+\widetilde{S}_1 \mathbb E[Z]+\widetilde{S}_2\mathbb E[u]+\mathbb E[q]\right]dt+d\eta.
\end{aligned}
\end{equation*}
Through some straightforward  calculations, we derive
\begin{equation*}
\left\{\begin{aligned}
&dY=\ \left[AY+\widehat{A}\mathbb E[Y]+Bu+\widehat{B}\mathbb E[u]+CZ+\widehat{C}\mathbb E[Z]+f\right]dt+ZdW,
\\&Y(T)= \zeta.
\end{aligned}
\right.
\end{equation*}
Moreover,
$$B^\top X+\widehat{B}^\top \mathbb E[X]+S_2^\top Y+\widehat{S}_2^\top \mathbb E[Y]+R_{22}u+\widehat{R}_{22}\mathbb E[u]+\rho_2=0.$$
Theorem \ref{necessary condition} implies that the optimal control of  Problem (BLQ) is given by \eqref{optimal control}. In the following, we prove that the value function of Problem (BLQ) is given by \eqref{eq: optimal cost fuctional}. Indeed,
\begin{equation*}
\begin{aligned}
V(\zeta)
=&\frac{1}{2}\mathbb E\left[2\langle g,  Y(0)\rangle+\int_{0}^T\Big(2\langle q,Y\rangle+2\langle\rho_1, Z\rangle+2\langle\rho_2,u\rangle+2\langle S_1Z,  Y\rangle+2\langle S_2 u,  Y\rangle+\langle R_{11}Z, Z\rangle\right.\\&\left.+\langle R_{22}u, u\rangle
+2\langle \widehat{S}_1\mathbb E[Z],  \mathbb E[Y]\rangle+2\langle \widehat{S}_2 \mathbb E[u],  \mathbb E[Y]\rangle+\langle \widehat{R}_{11}\mathbb E[Z],  \mathbb E[Z]\rangle+\langle \widehat{R}_{22}\mathbb E[u], \mathbb E[u]\rangle\Big)dt\right]
\\=&\frac{1}{2}\mathbb E\left[\langle g, Y(0)\rangle-\langle X(T), Y(T)\rangle+\int_{0}^T\Big(\langle f,X\rangle+\langle q,Y\rangle+\langle\rho_1, Z\rangle+\langle\rho_2,u\rangle\Big)dt\right].
\end{aligned}
\end{equation*}
From equations \eqref{adjoint decouple} and \eqref{eq.BSDE1}, we have
\begin{equation*}
\begin{aligned}
&\mathbb E[\langle X(T), Y(T)\rangle]\\=&\mathbb E[\langle X(T), \eta(T)\rangle]
\\=&\mathbb E[\langle X(0), \eta(0)\rangle]+\mathbb E\left[\int_0^T\langle X, f+\Upsilon(q-\mathbb E[q])+\widetilde{\Upsilon}\mathbb E[q]-(C+\Upsilon S_1)(I+\Upsilon R_{11})^{-1}\Upsilon(\rho_1-\mathbb E[\rho_1])\right.\\&-(B+\Upsilon S_2) R_{22}^{-1}(\rho_2-\mathbb E[\rho_2])
-(\widetilde{C}+\widetilde{\Upsilon} \widetilde{S}_1)(I+\Upsilon \widetilde{R}_{11})^{-1}\Upsilon\mathbb E[\rho_1]-(\widetilde{B}+\widetilde{\Upsilon} \widetilde{S}_2) \widetilde{R}_{22}^{-1}\mathbb E[\rho_2]\rangle
\\&+\langle\eta,\left(S_1(I+\Upsilon R_{11})^{-1}\Upsilon S_1^\top+S_2R_{22}^{-1}S_2^\top\right)(\eta-\mathbb E[\eta])+\left(\widetilde{S}_1(I+\Upsilon \widetilde{R}_{11})^{-1}\Upsilon \widetilde{S}_1^\top+\widetilde{S}_2\widetilde{R}_{22}^{-1}\widetilde{S}_2^\top\right)\mathbb E[\eta]\rangle
\\&+\langle \eta,S_1(I+\Upsilon R_{11})^{-1}[\Upsilon(\rho_1-\mathbb E[\rho_1])-2(\beta-\mathbb E[\beta])]+S_2R_{22}^{-1}(\rho_2-\mathbb E[\rho_2])
\\&+\widetilde{S}_1(I+\Upsilon \widetilde{R}_{11})^{-1}[\Upsilon\mathbb E[\rho_1]-2\mathbb E[\beta]]+\widetilde{S}_2\widetilde{R}_{22}^{-1}\mathbb E[\rho_2]-q\rangle
\\&-\langle\beta,(I+R_{11}\Upsilon)^{-1}[(\rho_1-\mathbb E[\rho_1])+R_{11}(\beta-\mathbb E[\beta])]+(I+\widetilde{R}_{11}\Upsilon)^{-1}(\mathbb E[\rho_1]+\widetilde{R}_{11}\mathbb E[\beta])\rangle dt\Big].
\end{aligned}
\end{equation*}
Combining the above equations, we obtain $V(\zeta)$ satisfies \eqref{eq: optimal cost fuctional}.
\end{proof}

\section{Sufficiency of  Riccati equations}
In the above, we construct the optimal control of Problem (BLQ) under Assumption $A3$. However, it is not easy to determine whether Assumption $A3$ is satisfied for a general LQ control problem of mean-field BSDE. In this section, we give a sufficient condition, which ensures the  uniform convexity of cost functional $J(\zeta;u)$. 
\begin{theorem}\label{th: sufficient}
Let Assumptions $A1$-$A2$ and  \eqref{eq: simcoeff} hold. If $R_{22}\gg 0,  \widetilde{R}_{22}\gg 0$, Riccati equations \eqref{eq.RE1} and  \eqref{eq.RE2} admit unique solutions $\Upsilon\geq 0, \widetilde{\Upsilon}\geq 0$ 
such that $I+\Upsilon R_{11}, I+\Upsilon \widetilde{R}_{11}$ are invertible on $[0, T]$,  $(I+\Upsilon R_{11})^{-1}, (I+\Upsilon \widetilde{R}_{11})^{-1}\in\mathcal L^\infty(0,T;\mathbb R^{n\times n})$ and $(I+\Upsilon R_{11})^{-1}\Upsilon\geq 0, (I+\Upsilon \widetilde{R}_{11})^{-1}\Upsilon\geq 0$, then there exists a constant $\alpha>0$, such that
\begin{align*}
J_0(0;  u)\geq \alpha \mathbb E\left[\int_{0}^T|u|^2dt\right],   \forall u\in \mathcal U[0, T].
\end{align*}
\end{theorem}

In order to prove Theorem \ref{th: sufficient}, we need the following results.
\begin{lemma}\label{lemma: uniform convex}
Let Assumptions $A1$-$A2$ hold. For any $u\in \mathcal U[0, T]$,
let $(y, z)$ be the solution of
\begin{equation*}
\left\{ \begin{aligned}
dy=&\Big(Ay+\widehat{A}\mathbb E[y]+Bu+\widehat{B}\mathbb E[u]+Cz+\widehat{C}\mathbb E[z]\Big)dt+zdW,  \\
y(T)=&\ 0.
\end{aligned}\right.
\end{equation*}
Then for any $\Theta,  \widetilde{\Theta}\in \mathcal{L}^\infty(0,  T; \mathbb R^{m\times n})$, there exists a constant $\gamma>0$, such that
\begin{equation}\label{eq: uniform convex}
\begin{aligned}
\mathbb E\left[\int_{0}^T|u-\Theta(y-\mathbb E[y])|^2dt\right]&\geq \gamma\mathbb E\left[\int_{0}^T|u|^2dt\right],
\\ \int_{0}^T|\mathbb E[u]-\widetilde{\Theta}\mathbb E[y]|^2dt&\geq \gamma\left[\int_{0}^T|\mathbb E[u]|^2dt\right].
\end{aligned}
\end{equation}
\end{lemma}
\begin{proof}
For any $\Theta\in \mathcal{L}^{\infty}(0,  T; \mathbb R^{m\times n})$,
we define a bounded linear operator $\mathcal A: \mathcal{L}_{\mathbb{F}}^2(0,  T; \mathbb R^{m})\to \mathcal{L}_{\mathbb{F}}^2(0,  T; \mathbb R^{m})$ by
$$\mathcal{A}u=u-\Theta(y-\mathbb E[y]).$$
Then $\mathcal{A}$ is bijective and its inverse $\mathcal{A}^{-1}$ is given by
$$\mathcal{A}^{-1}u=u+\Theta(\widetilde{y}-\mathbb E[\widetilde{y}]),$$
where
\begin{equation*}
\left\{ \begin{aligned}
d\widetilde{y}=&\left((A+B\Theta)\widetilde{y}+(\widehat{A}-B\Theta)\mathbb E[\widetilde{y}]+Bu+\widehat{B}\mathbb E[u]+C\widetilde{z}+\widehat{C}\mathbb E[\widetilde{z}]\right)dt+\widetilde{z}dW,  \\
\widetilde{y}(T)=&\ 0.
\end{aligned}\right.
\end{equation*}
By the bounded inverse theorem, $\mathcal{A}^{-1}$ is bounded with norm $||\mathcal{A}^{-1}||>0$. Thus,
\begin{equation*}
\begin{aligned}
\mathbb E\left[\int_0^T|u(t)|^2dt\right]&=\mathbb E\left[\int_0^T|(\mathcal{A}^{-1}\mathcal{A}u)(t)|^2dt\right]\leq ||\mathcal{A}^{-1}||^2\mathbb E\left[\int_0^T|(\mathcal{A}u)(t)|^2dt\right]
\\&=||\mathcal{A}^{-1}||^2\mathbb E\left[\int_0^T|u-\Theta(y-\mathbb E[y])|^2dt\right],  \ \ \forall u\in\mathcal{L}_{\mathbb{F}}^2(0,  T; \mathbb R^{m}),
\end{aligned}
\end{equation*}
which implies the first inequality in \eqref{eq: uniform convex} with $\gamma=||\mathcal{A}^{-1}||^2$.

To prove the second inequality, for any $\widetilde{\Theta}\in \mathcal{L}^{\infty}(0,  T; \mathbb R^{m\times n})$, we define a bounded linear operator $\widetilde{\mathcal A}: \mathcal L_{\mathbb F}^2(0, T;\mathbb{R}^m)\to \mathcal L_{\mathbb F}^2(0, T;\mathbb{R}^m)$ by
$$\widetilde{\mathcal A}u=u-\widetilde{\Theta}\mathbb E[y].$$
Then $\widetilde{\mathcal A}$ is bijective and its inverse $\widetilde{\mathcal A}^{-1}$ is given by
$$\widetilde{\mathcal A}^{-1}u=u+\widetilde{\Theta}\mathbb E[\widehat{y}], $$
where
\begin{equation*}
\left\{ \begin{aligned}
d\widehat{y}=&\left[A\widehat{y}+\Big(\widehat{A}+\widetilde{B}\widetilde{\Theta}\Big)\mathbb E[\widehat{y}]
+Bu+\widehat{B}\mathbb E[u]+C\widehat{z}+\widehat{C}\mathbb E[\widehat{z}]\right]dt+\widehat{z}dW,  \\
\widehat{y}(T)=&\ 0.
\end{aligned}\right.
\end{equation*}
By the bounded inverse theorem, $\widetilde{\mathcal A}^{-1}$ is bounded with norm $||\widetilde{\mathcal A}^{-1}||>0$. Thus,
\begin{equation*}
\begin{aligned}
\int_0^T\big|\mathbb E[u(t)]\big|^2dt&=\int_0^T\big|(\mathbb E[\widetilde{\mathcal A}^{-1}\widetilde{\mathcal A}u])(t)\big|^2dt\leq||\widetilde{\mathcal A}^{-1}||^2\mathbb \int_0^T\big| (\mathbb E[\widetilde{\mathcal A}u])(t)\big|^2dt
\\&=||\widetilde{\mathcal A}^{-1}||^2\int_0^T\big|\mathbb E[u]-\widetilde{\Theta}\mathbb E[y]\big|^2dt,  \ \ \forall u\in\mathcal{L}_{\mathbb F}^2(0,  T; \mathbb R^{m}),
\end{aligned}
\end{equation*}
which implies the second inequality in \eqref{eq: uniform convex} with $\gamma=||\widetilde{\mathcal{A}}^{-1}||^2$.
\end{proof}

We now proceed to prove Theorem \ref{th: sufficient}.
\begin{proof}
By the continuity theorem of solution on initial condition,  there exists $\lambda_2$, such that for any $\lambda\geq \lambda_2$, Riccati equations
\begin{equation*}
\left\{
\begin{aligned}
&\dot\Upsilon_{\lambda}-\Upsilon_{\lambda} A^\top-A\Upsilon_{\lambda}
\\&+\left(\begin{array}{c}C^\top+ S_1^\top\Upsilon_{\lambda}\\B^\top+S_2^\top\Upsilon_{\lambda}\end{array}\right)^\top\left(\begin{array}{cc}I+\Upsilon_{\lambda} R_{11}& 0\\0&R_{22}\end{array}\right)^{-1}\left(\begin{array}{c}\Upsilon_{\lambda} C^\top+\Upsilon_{\lambda} S_1^\top\Upsilon_{\lambda}\\B^\top+S_2^\top\Upsilon_{\lambda}\end{array}\right)=0, \\
&\Upsilon_{\lambda}(T)=\lambda^{-1}I,
\end{aligned}
\right.
\end{equation*}
and
\begin{equation*}
\left\{
\begin{aligned}
&\dot{\widetilde{\Upsilon}}_{\lambda}-\widetilde{\Upsilon}_{\lambda} \widetilde{A}^\top-\widetilde{A}\widetilde{\Upsilon}_{\lambda}
\\&+\left(\begin{array}{c}\widetilde{C}^\top+\widetilde{S}_1^\top\widetilde{\Upsilon}_{\lambda}
\\\widetilde{B}^\top+\widetilde{S}_2^\top\widetilde{\Upsilon}_{\lambda}\end{array}\right)^\top\left(\begin{array}{cc}I+\Upsilon_{\lambda} \widetilde{R}_{11}&0\\0&\widetilde{R}_{22}\end{array}\right)^{-1}\left(\begin{array}{c}\Upsilon_{\lambda} \widetilde{C}^\top+\Upsilon_{\lambda}  \widetilde{S}_1^\top\widetilde{\Upsilon}_{\lambda}
\\\widetilde{B}^\top+\widetilde{S}_2^\top\widetilde{\Upsilon}_{\lambda}\end{array}\right)=0, \\
&\widetilde{\Upsilon}_{\lambda}(T)=\lambda^{-1}I
\end{aligned}
\right.
\end{equation*}
admit  unique solutions $\Upsilon_{\lambda}, \widetilde{\Upsilon}_{\lambda}$.
Moreover, by the comparison theorem, we further have $\Upsilon_{\lambda}>\Upsilon\geq 0, \widetilde{\Upsilon}_{\lambda}>\widetilde{\Upsilon}\geq 0$, and $\Upsilon_{\lambda}, \widetilde{\Upsilon}_{\lambda}$ are monotonically decreasing with respect to $\lambda$.
In the following, we prove that $(I+\Upsilon_{\lambda} R_{11})^{-1}\Upsilon_{\lambda}\geq 0$, $(I+\Upsilon_{\lambda} \widetilde{R}_{11})^{-1}\Upsilon_{\lambda}\geq 0$. Actually,
\begin{equation*}
\begin{aligned}
&(I+\Upsilon_{\lambda} R_{11})^{-1}\Upsilon_{\lambda}=(R_{11}+\Upsilon_{\lambda}^{-1})^{-1},
\\&(I+\Upsilon_{\lambda} \widetilde{R}_{11})^{-1}\Upsilon_{\lambda}=(\widetilde{R}_{11}+\Upsilon_{\lambda}^{-1})^{-1},
\end{aligned}
\end{equation*}
which imply that $(I+\Upsilon_{\lambda} R_{11})^{-1}\Upsilon_{\lambda}$ and  $(I+\Upsilon_{\lambda} \widetilde{R}_{11})^{-1}\Upsilon_{\lambda}$ are monotonically decreasing with respect to $\lambda$. Combining with the fact that $(I+\Upsilon R_{11})^{-1}\Upsilon\geq 0, (I+\Upsilon \widetilde{R}_{11})^{-1}\Upsilon\geq 0$, we have
$(I+\Upsilon_{\lambda} R_{11})^{-1}\Upsilon_{\lambda}\geq 0$, $(I+\Upsilon_{\lambda} \widetilde{R}_{11})^{-1}\Upsilon_{\lambda}\geq 0$.

Defining $\Pi_{\lambda}=\Upsilon_{\lambda}^{-1}, \widetilde{\Pi}_{\lambda}=\widetilde{\Upsilon}_{\lambda}^{-1}$ for $\lambda\geq \lambda_2$, then $\Pi_{\lambda}> 0,\widetilde{\Pi}_{\lambda}>0.$
It is clear that  $\Pi_{\lambda}$ and $\widetilde{\Pi}_{\lambda}$ satisfy Riccati equations
\begin{equation}
\left\{\begin{aligned}&\dot{\Pi}_{\lambda}+\Pi_{\lambda}A+A^\top \Pi_{\lambda}\\&\ \ -\left(\begin{array}{c}
C^\top\Pi_{\lambda}+S_1^\top\\B^\top\Pi_{\lambda}+S_2^\top\end{array}\right)^\top\left(\begin{array}{cc}
R_{11}+\Pi_{\lambda}&0\\0&R_{22}\end{array}\right)^{-1}\left(\begin{array}{c}
C^\top \Pi_{\lambda}+S_1^\top\\B^\top \Pi_{\lambda}+S_2^\top\end{array}\right)=0,\\&\Pi_{\lambda}(T)=\lambda I,
\end{aligned}\right.
\end{equation}
and
\begin{equation}
\left\{\begin{aligned}&\dot{\widetilde{\Pi}}_{\lambda}+\widetilde{\Pi}_{\lambda}\widetilde{A}+\widetilde{A}^\top \widetilde{\Pi}_{\lambda}\\&\ \ -\left(\begin{array}{c}
\widetilde{C}^\top\widetilde{\Pi}_{\lambda}+\widetilde{S}_1^\top\\\widetilde{B}^\top \widetilde{\Pi}_{\lambda}+\widetilde{S}_2^\top\end{array}\right)^\top\left(\begin{array}{cc}
\widetilde{R}_{11}+\Pi_{\lambda}&0\\0&\widetilde{R}_{22}\end{array}\right)^{-1}\left(\begin{array}{c}
\widetilde{C}^\top \widetilde{\Pi}_{\lambda}+\widetilde{S}_1^\top\\\widetilde{B}^\top \widetilde{\Pi}_{\lambda}+\widetilde{S}_2^\top\end{array}\right)=0,\\&\widetilde{\Pi}_{\lambda}(T)=\lambda I,
\end{aligned}
\right.
\end{equation}
respectively. Moreover,
$R_{11}+\Pi_{\lambda}\gg0, \widetilde{R}_{11}+\Pi_{\lambda}\gg 0.$

Let $$\Sigma=\big\{\psi\in \mathcal{L}^\infty(0,T;\mathbb S^n)|\ \psi
\text{ is  a deterministic continuous differential function}\big\}.$$
For  $h=(H, \widetilde{H})\in \Sigma\times \Sigma$, we define
\begin{equation}\label{eq:equicoeff}
\left\{
\begin{aligned}
&Q_{h}=\dot{H}+HA+A^\top H, \;  S_{1,  h}=S_1+H C, \;  S_{2,  h}=S_2+H B,\;  N_{1,  h}=R_{11}+H,
\\&\widetilde{Q}_{h}=\dot{\widetilde{H}}+\widetilde{H}\widetilde{A}+\widetilde{A}^\top \widetilde{H},\;   \widetilde{S}_{1,  h}=\widetilde{S}_1+\widetilde{H} \widetilde{C},  \; \widetilde{S}_{2,  h}=\widetilde{S}_2+\widetilde{H} \widetilde{B}, \; \widetilde{N}_{1,  h}=\widetilde{R}_{11}+H,
\\&\widetilde{q}_h=q+H f+(\widetilde{H}-H) \mathbb E[f].
\end{aligned}\right.
\end{equation}
We introduce a family of equivalent cost functionals
\begin{equation*}
\begin{aligned}
&J_h(\zeta;  u)\\=&\ \frac{1}{2}\mathbb E\left[\int_{0}^T\Bigg\langle\left(\begin{array}{ccc}
Q_{h}& S_{1,  h}& S_{2,  h}
\\S_{1,  h}^\top &N_{1,  h}&0
\\S_{2,  h}^\top&0&R_{22}\end{array}\right)\left(\begin{array}{c}Y-\mathbb E[Y]\\Z-\mathbb E[Z]\\u-\mathbb E[u]\end{array}\right),  \left(\begin{array}{c}Y-\mathbb E[Y]\\Z-\mathbb E[Z]\\u-\mathbb E[u]\end{array}\right)\Bigg\rangle dt\right.
\\&+\int_{0}^T\Bigg\langle\left(\begin{array}{ccc}
\widetilde{Q}_{h}& \widetilde{S}_{1,  h}& \widetilde{S}_{2,  h}
\\\widetilde{S}_{1,  h}^\top&\widetilde{N}_{1,  h}&0
\\\widetilde{S}_{2,  h}^\top&0&\widetilde{R}_{22} \end{array}\right)\left(\begin{array}{c}\mathbb E[Y]\\\mathbb E[Z]\\\mathbb E[u]\end{array}\right),  \left(\begin{array}{c}\mathbb E[Y]\\\mathbb E[Z]\\\mathbb E[u]\end{array}\right)\Bigg\rangle dt
\\&\left.+2\int_{0}^T\Bigg\langle\left(\begin{array}{c}\widetilde{q}\\\rho_1\\\rho_2\end{array}\right),  \left(\begin{array}{c}Y\\Z\\u\end{array}\right)\Bigg\rangle dt
+2\langle g,  Y(0)\rangle\right]
\\&+\frac{1}{2}\mathbb E\left[\langle H(0)(Y(0)-\mathbb E[Y(0)]),  Y(0)-\mathbb E[Y(0)]\rangle
+\langle \widetilde{H}(0)\mathbb E[Y(0)],  \mathbb E[Y(0)]\rangle\right].
\end{aligned}
\end{equation*}
Actually, we have
\begin{equation}\label{eq:equi}
\begin{aligned}
J(\zeta;  u)=J_{h}(\zeta;  u)-\frac{1}{2}\mathbb E\left[\langle H(T)(\zeta-\mathbb E[\zeta]),  \zeta-\mathbb E[\zeta]\rangle +\langle\widetilde{ H}(T)\mathbb E[\zeta], \mathbb E[\zeta]\rangle\right].
\end{aligned}
\end{equation}
Thus we may take   $h=(H,\widetilde{H})=(\Pi_{\lambda},\widetilde{\Pi}_{\lambda})$ with  $\lambda\geq \lambda_2$. It is obviously that
\begin{equation*}
\begin{aligned}
&J_0(0;  u)=J_{0,h}(0;  u)
\\=&\ \frac{1}{2}\mathbb E\left[\int_{0}^T\Bigg\langle\left(\begin{array}{ccc}
Q_{h}& S_{1,  h}& S_{2,  h}
\\S_{1,  h}^\top &N_{1,  h}&0
\\S_{2,  h}^\top&0&R_{22}\end{array}\right)\left(\begin{array}{c}y-\mathbb E[y]\\z-\mathbb E[z]\\u-\mathbb E[u]\end{array}\right),  \left(\begin{array}{c}y-\mathbb E[y]\\z-\mathbb E[z]\\u-\mathbb E[u]\end{array}\right)\Bigg\rangle dt\right.
\\&\left.+\int_{0}^T\Bigg\langle\left(\begin{array}{ccc}
\widetilde{Q}_{h}& \widetilde{S}_{1,  h}& \widetilde{S}_{2,  h}
\\\widetilde{S}_{1,  h}^\top&\widetilde{N}_{1,  h}&0
\\\widetilde{S}_{2,  h}^\top&0&\widetilde{R}_{22} \end{array}\right)\left(\begin{array}{c}\mathbb E[y]\\\mathbb E[z]\\\mathbb E[u]\end{array}\right),  \left(\begin{array}{c}\mathbb E[y]\\\mathbb E[z]\\\mathbb E[u]\end{array}\right)\Bigg\rangle dt\right]
\\&+\frac{1}{2}\mathbb E\left[\langle H(0)(y(0)-\mathbb E[y(0)]),  y(0)-\mathbb E[y(0)]\rangle
+\langle \widetilde{H}(0)\mathbb E[y(0)],  \mathbb E[y(0)]\rangle\right].
\end{aligned}
\end{equation*}
Recalling that $R_{22}\gg0, \widetilde{R}_{22}\gg0$ and combining with Lemma \ref{lemma: uniform convex}, we have
\begin{equation*}
\begin{aligned}
&J_0(0;  u)
\\\geq&\ \frac{1}{2}\mathbb E\left[\int_{0}^T\bigg(\left\langle R_{22}\left(u-\mathbb E[u]+R_{22}^{-1}S_{2, h}^{\top}(y-\mathbb E[y])\right), u-\mathbb E[u]+R_{22}^{-1}S_{2, h}^{\top}(y-\mathbb E[y])\right\rangle\right.
\\&\left.+\left\langle \widetilde{R}_{22}\left(\mathbb E[u]+\widetilde{R}_{22}^{-1}\widetilde{S}_{2, h}^{\top}\mathbb E[y]\right), \mathbb E[u]+\widetilde{R}_{22}^{-1}\widetilde{S}_{2, h}^{\top}\mathbb E[y]\right\rangle\bigg) dt\right]
\\\geq&\ \frac{1}{2}\delta\mathbb E\left[\int_{0}^T\Big(\left|u-\mathbb E[u]+R_{22}^{-1}S_{2, h}^{\top}(y-\mathbb E[y])\right|^2+\gamma|\mathbb E[u]|^2\Big)\right]
\\\geq&\ \frac{\delta\gamma}{2(1+\gamma)}\mathbb E\left[\int_{0}^T\left|u+R_{22}^{-1}S_{2, h}^{\top}(y-\mathbb E[y])\right|^2\right]
\\\geq &\ \frac{\delta\gamma^2}{2(1+\gamma)}\mathbb E\left[\int_{0}^T\left|u\right|^2\right].
\end{aligned}
\end{equation*}
The proof is completed.
\end{proof}
\section{Examples}
In this section, we present two illustrative examples. In the first example, the assumptions (H2) in \cite{LSX2019} does not hold, but the corresponding Riccati equations admit unique solutions which satisfy  Theorem \ref{th: sufficient}. Thus, the cost functional is uniformly convex.
This example shows that the uniform convexity condition (Assumption  A3)   is indeed weaker than assumptions (H2) in \cite{LSX2019}. In Example 5.2, it is difficult to prove the existence
and uniqueness of solutions to related Riccati equations. We use the equivalent cost functional
method to construct an  equivalent functional which satisfies Assumption A3 first, and then we
obtain an optimal control of the original stochastic control problem via solutions of Riccati
equations.

\textbf{Example 6.1} Consider a one-dimensional controlled mean-field BSDE
\begin{equation*}\label{example 1}
\left\{ \begin{aligned}
dY=&\left(2Y-2\mathbb E[Y]+2u-\mathbb E[u]\right)dt+ZdW,\quad t\in{[0,1]},\\
Y(1)=&\ \xi,
\end{aligned}\right.
\end{equation*}
with cost functional
\begin{equation*}\label{cost functional example1}
\begin{aligned}
&J(\xi;u)=\frac{1}{2}\mathbb E\Bigg[\int_0^1\Big(-Z^2+2u^2-2\mathbb E[Y]\mathbb E[u]-\mathbb E[Z]^2-\mathbb E[u]^2\Big)dt\Bigg].
\end{aligned}
\end{equation*}
It is difficult to check that whether $J(0;u)\geq\alpha\mathbb E\big[\int_0^T|u|^2 dt\big]$ holds for some $\alpha>0$. 
With the data, Riccati equations \eqref{eq.RE1} and  \eqref{eq.RE2}  are
\begin{equation*}
\left\{\begin{aligned}
&\dot{\Upsilon}-4\Upsilon+2=0,\\
&\Upsilon(1)=0,
\end{aligned}
\right.
\end{equation*}
and
\begin{equation*}
\left\{\begin{aligned}
&\dot{\widetilde{\Upsilon}}+(1-\widetilde{\Upsilon})^2
=0,\\
&\widetilde{\Upsilon}(1)=0,
\end{aligned}
\right.
\end{equation*}
respectively. Solving them, we get  $$\Upsilon(t)=\frac{1}{2}-\frac{\exp(4t-4)}{2}, \ \widetilde{\Upsilon}(t)=\frac{1}{t-2}+1.$$
Note that
\begin{equation*}
\begin{aligned}\Upsilon\geq 0,  \widetilde{\Upsilon}\geq 0, 1-\Upsilon=\frac{1}{2}+\frac{\exp(4t-4)}{2},1-2\Upsilon =\exp(4t-4).
\end{aligned}
\end{equation*}
Theorem \ref{th: sufficient} implies  that there exists a constant $\alpha>0$, such that
$$J(0;  u)\geq \alpha \mathbb E\left[\int_{0}^T|u|^2dt\right],   \forall u\in \mathcal U[0, T].$$
We now introduce the following mean-field BSDE:
\begin{equation}
\left\{
\begin{aligned}
&d\eta=\Big[2(\eta-\mathbb E [\eta])+ (1-\widetilde{\Upsilon})\mathbb E [\eta]\Big]dt+\beta dW, \\
&\eta(T)=\ \zeta.
\end{aligned}
\right.
\end{equation}
According Theorem \ref{Th: no-feedback},
  the optimal control   is given by
\begin{equation}\label{optimal control-exa}
\begin{aligned}
u=&-X+\widetilde{\Upsilon}\mathbb E[X] +\mathbb E[\eta],
\end{aligned}
\end{equation}
where $X$ is given by
\begin{equation}\label{adjoint decouple-exa}
\left\{\begin{aligned}
&dX=-\bigg\{2(X-\mathbb E[X])
+(1-\widetilde{\Upsilon})\mathbb E[X]
-\mathbb E[\eta]\bigg\}dt
\\&\qquad+\bigg\{(I-\Upsilon)^{-1}(\beta-\mathbb E[\beta])
+2(I-2\Upsilon)^{-1}\mathbb E[\beta]
\bigg\}dW,
\\&X(0)=0.
\end{aligned}
\right.
\end{equation}
Moreover, the corresponding value function is
\begin{equation*}
\begin{aligned}
&V(\zeta)=-\frac{1}{2}\mathbb E\left\{\int_{0}^T\left[\left\langle \beta,  (I-\Upsilon )^{-1}(\beta-\mathbb E[\beta])+2(I-2\Upsilon)^{-1}E[\beta]\right\rangle
+\Big\langle \eta,\mathbb E[\eta]\Big\rangle
\right]dt\right\}.
\end{aligned}
\end{equation*}

\textbf{Example 6.2} Consider a one-dimensional controlled mean-field BSDE
\begin{equation*}\label{example 1}
\left\{ \begin{aligned}
dY=&\left(Y+\mathbb E[Y]+u+\mathbb E[u]+Z\right)dt+ZdW,\quad t\in{[0,1]},\\
Y(1)=&\ \xi,
\end{aligned}\right.
\end{equation*}
with cost functional
\begin{equation*}\label{cost functional example1}
\begin{aligned}
&J(\xi;u)=\frac{1}{2}\mathbb E\Bigg[\int_0^1\Big(-8YZ-6Yu-2Z^2+u^2
+4\mathbb E[Y]\mathbb E[Z]-2\mathbb E[Y]\mathbb E[u]+\mathbb E[Z]^2\Big)dt\Bigg].
\end{aligned}
\end{equation*}
It is difficult to check that whether an optimal control exists.
With the data, Riccati equations \eqref{eq.RE1} and  \eqref{eq.RE2}  are
\begin{equation}\label{eq.APP2RE1}
\left\{\begin{aligned}
&\dot{\Upsilon}-2\Upsilon+\frac{\Upsilon(1-4\Upsilon)^2}{1-2\Upsilon}+(1-3\Upsilon)^2=0,\\
&\Upsilon(1)=0,
\end{aligned}
\right.
\end{equation}
and
\begin{equation}\label{eq.APP2RE2}
\left\{\begin{aligned}
&\dot{\widetilde{\Upsilon}}-4\widetilde{\Upsilon}
+\frac{\Upsilon(1-2\widetilde{\Upsilon})^2}{1-\Upsilon}
+(2-4\widetilde{\Upsilon})^2
=0,\\
&\widetilde{\Upsilon}(1)=0,
\end{aligned}
\right.
\end{equation}
respectively. It is difficult to give the solvabilities of   Riccati equations \eqref{eq.APP2RE1} and \eqref{eq.APP2RE2} due to the complexity.  According to \eqref{eq:equicoeff},
we have for $h_0=(H_0, \widetilde{H}_0)=(3,2)$,
\begin{equation}
\left\{
\begin{aligned}
&Q_{h_0}=6, \;  S_{1,  h}=-1, \;  S_{2,  h}=0,\;  N_{1,  h}=1,
\\&\widetilde{Q}_{h}=8,\;   \widetilde{S}_{1,  h}=0,  \; \widetilde{S}_{2,  h}=0, \; \widetilde{N}_{1,  h}=1.
\end{aligned}\right.
\end{equation}
We can check that
$$J_{h_0}(0;u)\geq\alpha \mathbb E\left[\int_{0}^T|u|^2dt\right],  \forall u\in \mathcal U[0,T].$$
According to \eqref{eq:equi}, we have
$$J(0;u)=J_{h_0}(0;u)\geq\alpha \mathbb E\left[\int_{0}^T|u|^2dt\right],  \forall u\in \mathcal U[0,T],$$
which implies that Assumption A3 holds. Thus we obtain that  Riccati equations  \eqref{eq.APP2RE1} and \eqref{eq.APP2RE2} admits unique solutions $\Upsilon$ and $\widetilde{\Upsilon}$, respectively from Theorem \ref{Th: solva RE backward}.
Further, mean-field BSDE \eqref{eq.BSDE1} can be rewritten as
\begin{equation*}
\left\{
\begin{aligned}
d\eta=&\Big[\eta+\mathbb E [\eta]+(1-4\Upsilon)(1-2\Upsilon)^{-1}\left(4\Upsilon (\eta-\mathbb E[\eta])+(\beta-\mathbb E[\beta])\right)
+3(1-3\Upsilon)\left(\eta-\mathbb E[\eta]\right)
\\&\ \ +\left[(1-2\widetilde{\Upsilon})(I-\Upsilon )^{-1}(2\Upsilon \mathbb E[\eta]+\mathbb E [\beta])+4(2-4\widetilde{\Upsilon})\mathbb E [\eta]
\right]dt+\beta dW, \\
\eta(T)=&\ \zeta.
\end{aligned}
\right.
\end{equation*}
According to Theorem \ref{Th: no-feedback}, the corresponding optimal control  is given by
\begin{equation*}
\begin{aligned}
u=&-\left[(1 -3\Upsilon)(X-\mathbb E[X])-3(\eta-\mathbb E[\eta])\right]-\left[(2-4\widetilde{\Upsilon})\mathbb E[X]-4\mathbb E[\eta]\right],
\end{aligned}
\end{equation*}
where $X$ satisfies
\begin{equation*}
\left\{\begin{aligned}
&dX=-\bigg\{\left[4-9\Upsilon+4\Upsilon(1-2\Upsilon)^{-1}(1-4\Upsilon)\right](X-\mathbb E[X])
\\&\qquad-4(I-2\Upsilon )^{-1}\left[4\Upsilon (\eta-\mathbb E[\eta])+(\beta-\mathbb E[\beta])\right]
-9(\eta-\mathbb E[\eta])
\\&\qquad+\left[2+2(I-\Upsilon)^{-1}\Upsilon(1-2\widetilde{\Upsilon})+4(2-4\widetilde{\Upsilon})\right]\mathbb E[X]
\\&\qquad-2(I-\Upsilon)^{-1}\left(2\Upsilon\mathbb E[\eta]+\mathbb E[\beta]\right)-16\mathbb E[\eta]\bigg\}dt
\\&\qquad-\bigg\{(I-2\Upsilon)^{-1}\left[(1-4\Upsilon)(X-\mathbb E[X])
-4(\eta-\mathbb E[\eta])
-2(\beta-\mathbb E[\beta])\right]
\\&\qquad+(I-\Upsilon)^{-1}\left[(1-2\widetilde{\Upsilon})\mathbb E[X]
-2\mathbb E[\eta]
-\mathbb E[\beta]\right]
\bigg\}dW,
\\&X(0)=0.  
\end{aligned}
\right.
\end{equation*}
Moreover, the corresponding value function is
\begin{equation*}
\begin{aligned}
V(\zeta)=&-\frac{1}{2}\mathbb E\left\{\int_{0}^T\left[\left\langle \beta,2(1-2\Upsilon)^{-1}(\beta-\mathbb E[\beta])+(1-\Upsilon)^{-1}E[\beta]\right\rangle\right.\right.
\\&+\Big\langle \eta,  \left(16\Upsilon(1-2\Upsilon)^{-1}+9\right)(\eta-\mathbb E[\eta])+\left(4\Upsilon(1-\Upsilon)^{-1}+16\right)\mathbb E[\eta]\Big\rangle
\\&\left.+2\left\langle\eta,  4(1-2\Upsilon)^{-1}(\beta-\mathbb E[\beta])+2(1-\Upsilon)^{-1}\mathbb E[\beta]\right\rangle\right]dt\bigg\}.
\end{aligned}
\end{equation*}
\section{Conclusion}
In this paper, we study an indefinite LQ control problem of mean-field BSDE. By using the limiting procedure and equivalent cost functional method, we propose some necessary and sufficient conditions for the uniform convexity of  cost functional are given in terms of two coupled Riccati equations with terminal conditions.  Further, we develop   general procedures for deriving explicit formulas of  optimal control and optimal cost  under the uniform convexity of  cost functional.  
The theoretical results obtained in this paper provide an insight into LQ zero-sum game problems of mean-field backward stochastic systems.
Well-posedness of Riccati equation plays a  crucial role in deriving the explicit representation  of saddle point.
Inspired by the general procedures developed in this paper, we may investigate the well-posedness of Riccati equation by establishing the relationship between forward and backward mean-field LQ zero-sum game problems. We will further  investigate some results on this problem and related topics in future.



\end{document}